\documentclass{amsart}
\usepackage{color}
\usepackage{graphics}
\usepackage{mathtools,tikz-cd}
\usepackage[pdftex,all]{xy}
\newtheorem{theorem}{Theorem}[section]
\newtheorem{lemma}[theorem]{Lemma}
\newtheorem{proposition}[theorem]{Proposition}
\newtheorem{corollary}[theorem]{Corollary}

\newtheorem{observation}[theorem]{Observation}

\theoremstyle{definition}
\newtheorem{definition}[theorem]{Definition}

\newtheorem{claim}{Claim}[theorem]
\newtheorem*{claim*}{Claim}
\newtheorem{observationInTh}{Observation}[theorem]

\newcommand{\of}[1]{\left(#1\right)}
\newcommand{\ofb}[1]{\left[#1\right]}
\newcommand{\ofc}[1]{\left\{#1\right\}}

\newcommand{\ofa}[1]{\left|#1\right|}

\newcounter{lcount}
\newcounter{lcount1}

\begin{document}
	
	\title[On conjugacy between natural extensions of 1-dimensional maps]{On conjugacy of natural extensions of 1-dimensional maps}
	\address[J. Boro\'nski]{
		AGH University of Science and Technology, Faculty of Applied Mathematics,
		al. Mickiewicza 30,
		30-059 Krak\'ow,
		Poland -- and -- National Supercomputing Centre IT4Innovations, University of Ostrava,
		IRAFM,
		30. dubna 22, 70103 Ostrava,
		Czech Republic}
	\email{boronski@agh.edu.pl}
	\address[P. Minc]{
		Department of Mathematics and Statistics, Auburn University, Auburn, AL 36849, USA}
	\email{mincpio@auburn.edu}
	\address[S. \v Stimac]{
		Department of Mathematics,
		Faculty of Science, University of Zagreb,
		Bijeni\v cka 30, 10\,000 Zagreb, Croatia}
	\email{sonja@math.hr}
	
	\subjclass[2020]{Primary 37E05, 37B45; Secondary 37D45, 20C20
	}
	
	\keywords{natural extension, inverse limit, dendrite, interval, pseudo-arc}
	
	\begin{abstract}
		We prove that for any nondegenerate dendrite $D$ there exist topologically mixing maps $F : D \to D$ and $f : [0, 1] \to [0, 1]$, such that
		the natural extensions (aka shift homeomorphisms) $\sigma_F$ and $\sigma_f$ are conjugate, and consequently the corresponding inverse limits are homeomorphic. Moreover, the map $f$ does not depend on the dendrite $D$, and can be selected so that the inverse limit
		$\underleftarrow{\lim} (D,F)$ is homeomorphic to the pseudo-arc. The result extends to any finite number of dendrites. Our work is motivated by, but independent of, the recent result of the first
		and third author on conjugation of Lozi and H\'enon maps to natural extensions of dendrite maps.
	\end{abstract}	
	
	\author{J.\ Boro\'nski}	
	\author{P.\ Minc}
	\author{S.\ \v{S}timac}
	\date{}
	\thanks{J. B. was supported by the National Science Centre, Poland (NCN), grant no. 2019/34/E/ST1/00237: ``Topological and
		Dynamical Properties in Parameterized Families of Non-Hyperbolic Attractors: the inverse limit approach''.\\
		S. \v{S}. was supported in part by the Croatian Science Foundation grant IP-2018-01-7491.}
	\maketitle
	\section{Introduction} The present paper pertains to the notion of the natural extension of a map, introduced by Rohlin in \cite{Rohlin}. Given a map $f:X\to X$ on a compact metric space $X$, the \emph{natural extension} of $f$ is the homeomorphism\footnote{In the mathematical literature this homeomorphism is also called the {\it shift} on the inverse limit $\underleftarrow{\lim} (X,f)$ and was used prior to Rohlin's work, for instance in an example considered by Williams \cite{WilliamsUn} . In our context, however, we want to emphasize the relation between noninvertible maps and their particular invertible extensions, and not merely consider a homeomorphism on the inverse limit space, hence the use of the term natural extension seems more appropriate.} $\sigma_f$ defined on the inverse limit space $\underleftarrow{\lim} (X,f)$ by $\sigma_f(x_0,x_1,x_2,\ldots)=(f(x_0),x_0,x_1,x_2,\ldots)$. It gives the unique invertible map semi-conjugate to $f$, such that any other invertible map semi-conjugate to $f$ is also semi-conjugate to $\sigma_f$. There exists a bijection between the set of invariant probability measures of $f$ and $\sigma_f$, and the topological entropies of $f$ and $\sigma_f$ coincide \cite{Rohlin}. Natural extensions of noninvertible maps of branched 1-manifolds appear in the mathematical literature in the context of studying dynamics on surfaces, e.g. in hyperbolic attractors \cite{Williams}, H\'enon attractors \cite{Barge}, \cite{Barge1}, \cite{Barge2}, $C^0$ dynamics \cite{BoylandBLMS}, \cite{BoylandGT}, \cite{Bruin}, complex dynamics \cite{Rempe-Gillen}, and rotation theory \cite{BOBirk}, \cite{BoylandInventiones}, \cite{Kwapisz}.

	Our paper is motivated by a recent result of the first and last author \cite{boronski-stimac}, in which it has been shown that for a class of mildly dissipative plane homeomorphisms, that contains positive Lebesgue measure subsets of Lozi and H\'enon maps, the dynamics on their attractors is conjugate to natural extensions of densely branching dendrite maps. In that context a question arose, whether these homeomorphisms could be also conjugate to natural extensions of maps on some simpler 1-dimensional spaces, such as the interval $[0,1]$. The homeomorphisms in question are transitive on their attractors, and sometimes even topologically mixing, and such properties are inherited by the respective dendrite maps. Therefore, it would seem as if the existence of dense orbits, together with density of the set of branch points in the dendrites, would force the corresponding inverse limit spaces to have a much richer topological structure than those of inverse limits of some simpler spaces, such as the interval, which has no branch points at all. This in turn would suggest that the above mentioned simplification is not possible. In the present paper, however, we show that such an intuition is deceitful. In Section \ref{factor} we introduce the notion of a {\it small folds property} for interval maps (Definition \ref{def:smallfolds}), and then show that every map with that property can be factored through an arbitrary dendrite. More precisely,  if $f:\ofb{0, 1}\to\ofb{0, 1}$ is a continuous surjection with the small folds property and $D$ is an arbitrary nondegenerate dendrite, then there are continuous surjections $g:\ofb{0,1}\to D$ and $h:D\to\ofb{0,1}$ such that $h\circ g=f$; see Lemma \ref{l:factf}.
	It follows that if $F=g\circ h$ then the natural extensions $\sigma_F$ and $\sigma_f$ are conjugate.
	In particular, $F$ is transitive on $D$ and $\underleftarrow{\lim} (D,F)$ is homeomorphic to the pseudo-arc if $f$ has the same properties on $\ofb{0,1}$. W.R.R. Transue and the second author of the present paper constructed a transitive map $f$ of $\ofb{0,1}$ onto itself such that $\varprojlim\of{\ofb{0,1},f}$ is homeomorphic to the pseudo-arc \cite{minc-transue} (see also \cite{Drwiega},\cite{Kawamura}, and \cite{KoscielniakOprochaTunchali} for related constructions). It is possible that this map has the small folds property, but it is not apparent how to prove it. However, in Section \ref{s:transitive+sfp} we tweak the original construction from \cite{minc-transue} to get a modified map $f$ that does have the small folds property in addition to the properties promised by \cite{minc-transue}, see Theorem \ref{t:smfolds}. This modified map $f$ can be factored through any nondegenerate dendrite $D$ creating interesting dynamics on $D$, see Theorem \ref{t:mainfactorization}.
 The following theorem is a restatement of Theorem \ref{t:mainfactorization}:
	\begin{theorem}\label{t:mainresult}
		For any nondegenerate dendrite $D$ there exist topologically mixing maps $F : D \to D$ and $f : [0, 1] \to [0, 1]$ such that the natural extensions $\sigma_F : \underleftarrow{\lim} (D,F) \to \underleftarrow{\lim} (D,F)$ and
		$\sigma_f : \underleftarrow{\lim} ([0, 1],f) \to \underleftarrow{\lim} ([0, 1],f)$ are conjugate. Moreover, the map $f$ does not depend on a dendrite $D$ and can be constructed so that $\underleftarrow{\lim} (D,F)$ is homeomorphic to the pseudo-arc.
	\end{theorem}
Note that the interval maps $f$ such that $\underleftarrow{\lim} ([0,1],f)$ is the pseudo-arc are generic in the closure of the subset of maps of the interval that have dense set of periodic points \cite{CO}. Moreover, all such maps have infinite topological entropy by \cite{Mouron} (see also \cite{BO} for a stronger result). Consequently, the same is true for the maps $F,\sigma_F$ and $\sigma_f$. This is noteworthy since, although any transitive interval map has positive entropy \cite{Blokh}, there do exist transitive zero entropy maps on dendrites \cite{KwietniakD} (see also \cite{OprochaD}, \cite{Stu}, \cite{HoehnD} and \cite{MouronMix} for related results). Note also that the class of dendrites is very rich. Every dendrite is locally connected, but there is a number of other properties with respect to which various elements of the class differ from each other, such as the properties of the subsets of end points and branch points. The set of end points in a dendrite can be finite, countably infinite or even uncountable, and either be closed or not. The set of branch points do not need to be finite, but can be countably infinite, and even dense in the dendrite. In addition, a branch point may separate the dendrite into infinitely many components. There exists a universal object in the class of all dendrites, the Wa\.zewski dendrite $D_\omega$ \cite{Wazewski}; i.e. any dendrite $D$ embeds as a closed subset of $D_\omega$. In that context, below we formulate a stronger version of our main result Theorem \ref{t:mfk}.
	\begin{theorem}
	For any $k\in\mathbb{N}$, and any dendrites $D_1,D_2,\ldots,D_k$, there exist topologically mixing maps $\{F_i : D_i \to D_i\}_{i=1}^{k}$ such that for any $i,j\in\{1,2,\ldots,k\}$ we have
	\begin{enumerate}
		\item $F_i$ and $F_j$ are semi-conjugate,
    \item the natural extensions $\sigma_{F_i} $ and
	$\sigma_{F_j}$ are conjugate, and
	\item  the inverse limits $\underleftarrow{\lim} (D_i,F_i)$ and $\underleftarrow{\lim} (D_j,F_j)$ are homeomorphic.
\end{enumerate}	
In addition, $F_1$ can be chosen so that $\underleftarrow{\lim} (D_i,F_i)$ is the pseudo-arc, for any $i=1,2,\ldots , k$.
\end{theorem}
	The above theorems produce, what seems to be, a very surprising family of examples for the question of conjugacy between natural extensions of self-maps of distinct dendrites. These examples, however, do not provide any new pieces of information for H\'enon or Lozi maps. Moreover, it seems rather unplausible that, for parameter values considered in \cite{boronski-stimac}, these maps would semi-conjugate to interval maps with small folds property, should any of them semi-conjugated to any interval map at all. In fact it is known that for certain H\'enon maps this is never true \cite{Barge}. Furthermore, H\'enon and Lozi attractors discussed in \cite{boronski-stimac} always contain nondegenerate arc components (such as branches of unstable manifolds), whereas the pseudo-arc contains no nondegenerate arcs at all. Recall that the {\it pseudo-arc} is a fractal-like object first constructed by Knaster in 1922. It was rediscovered by Moise in 1948 \cite{Moise}, who constructed it as a hereditarily equivalent continuum distinct from the arc, and in the same year by Bing who obtained it to show that there exist a topologically homogeneous\footnote{A space $X$ is topologically homogeneous if for any $y,z\in X$ there exists a homeomorphism $H:X\to X$ such that $H(y)=z$} plane continuum, distinct from the circle \cite{Bing}. Since then the pseudo-arc received a lot of attention in the mathematical literature, mainly in topology, but it also appears in other branches of mathematics, such as dynamical systems, including smooth and even complex dynamics; see e.g.\ \cite{Cheritat}, \cite{Handel}, \cite{Herman}, \cite{Paco}. Several topological characterizations of the pseudo-arc are known. One of the most recent ones, by Oversteegen and Hoehn \cite{HO} from 2016, states that the pseudo-arc is a unique topologically homogeneous plane nonseparating continuum (see also \cite{HO2}).
	
	\section{Preliminaries}\label{s:prelim}
	A map is a continuous function. Given a map $f:X\to X$ on a compact metric space $X$,
	we let
	\begin{equation}
	\underleftarrow{\lim} (X,f)
	=
	\{\big(x_0,x_1,\ldots,\big) \in X^{\mathbb{N}_0} :
	x_i\in X, x_i=f(x_{i+1}), \text{ for any }i \in {\mathbb{N}_0} \}
	\end{equation}
	and call $\underleftarrow{\lim} (X,f)$ the inverse limit of $X$ with bonding map $f$, or inverse limit of $f$ for short. It is equipped with
	metric induced from the
	\emph{product metric} in $X^{\mathbb{N}_0}$.
	The map $f$ is said to be {\it transitive} if for any two nonempty open sets $U,V\subset X$ there exists an $n\in\mathbb{N}$ such that $f^{n}(U)\cap V\neq\emptyset$. The map $f$ is said to be {\it topologically mixing} if for any two nonempty open sets $U,V\subset X$ there exists an $N\in\mathbb{N}$ such that $f^{n}(U)\cap V\neq\emptyset$ for all $n>N$. The map $f$ is said to be {\it topologically exact}, or {\it locally eventually onto}, if for every nonempty open set $U$ there exists an $n$ such that $f^n(U)=X$. It is evident from the definitions that topological exactness implies mixing, which implies transitivity.  A map $F:Y\to Y$ is said to be {\it semi-conjugate} to $f$ if there exists a surjective map $\varphi:Y\to X$ such that $f\circ \varphi=\varphi\circ F$. If in addition $\varphi$
	is a homeomorphism then $F$ is said to be {\it conjugate} to $f$. A {\it continuum} is a compact and connected metric space that contains at least two points.  A {\it dendrite} is a  locally connected continuum $D$ such that for all
	$x, y \in D$ there exists a unique (possibly degenerate)
	arc in $D$ with endpoints $x$ and $y$. We denote this arc by $xy$. The arcs $xy$ and $yx$ are the same as sets. We assume that $xy$ is oriented from $x$ to $y$ if this is needed. So, $x$ and $y$ are the first and the last points, respectively, of $xy$. An {\it end point} of $D$ is a point $e$ such that $D\setminus\{e\}$ is connected. The set of all end points of $D$ will by denoted by $E_D$. A {\it branch point} $b\in D$ is a point such that $D\setminus\{b\}$ has at least three components. For an arbitrary $x\in D$ and arbitrary positive number $\epsilon$, by $B_D\of{x,\epsilon}$ we will denote the open ball in $D$ with center at $x$ and radius $\epsilon$. In the present paper, a dendrite with finitely many branch points will be called a {\it tree}. An {\it arc} is a dendrite with no branch points.
	
	If $x$ and $y$ are real numbers, by $\ofb{x,y}$ we understand the closed interval between $x$ and $y$, regardless whether $x\le y$ or $x\ge y$. Similarly as in the case of dendrites, we use the order of endpoints to indicate the orientation of the interval. We do not use the notation $xy=\ofb{x,y}$ in the context of real numbers, even though $\ofb{x,y}$ is a dendrite.

	\section{Preliminary results on dendrites}\label{s:dendrite}
	
	G. T. Whyburn proved that every dendrite $D$ can be expressed as $D=E_D\cup\bigcup_{i=0}^{\infty}A_i$ where $\of{A_i}$ is a sequence of arcs such that $\lim_{i\to\infty}\operatorname{diam}\of{A_i}=0$; see \cite[V, (1.3)(iii), p.89]{Whyburn} and \cite[Corollary 10.28, p. 177]{Nadler}. Since we need a slightly stronger version of Whyburn's theorem we prove it below, see Theorem \ref{whyburn}.
	We start with the following simple observation.
	\begin{proposition}\label{p:AiTree}
		Let $T$ be a tree. Let $p_0\in E_T$ and $q_0,q_1,\dots,q_k$ be an enumeration of all points in $E_T\setminus\ofc{p_0}$. Then there exists a unique sequence of points $p_1,\dots,p_n\in T\setminus E_T$ such that, if $A_i=p_iq_i$ for all $i=0,\dots,k$, then
		$A_i\cap\bigcup_{j=0}^{i-1}A_j=\ofc{p_i}$ for each $i=1,\dots,k$.
		Moreover, $\bigcup_{j=0}^{k}A_j=T$.
	\end{proposition}
	Note that in the above proposition the points $p_1,\dots,p_n$ do not need to be distinct.
	
	Now let $D$ be a dendrite which is not a tree, and let $\mathcal{S}=(s_1,s_2,\dots)$ be a sequence of points dense in $D$.
	
	\begin{proposition}\label{prop:Ai}
		There exists an infinite sequence of nondegenerate arcs $A_0=p_0q_0, A_1=p_1q_1, A_2=p_2q_2,\dots$ contained in $D$ such that
		\begin{enumerate}
			\item $p_0,q_0\in E_D$, and
			\item for each integer $i\ge1$
			\begin{enumerate}
				\item $A_i\cap\bigcup_{j=0}^{i-1}A_j=\ofc{p_i}$,
				\item $q_i\in E_D$, and
				\item\label{prop:Ai:si} $s_i\in \bigcup_{j=0}^{i}A_j$.
			\end{enumerate}
		\end{enumerate}
		
	\end{proposition}
	
	Let $A_0=p_0q_0, A_1=p_1q_1, A_2=p_2q_2,\dots$ be as in the above proposition.
	
	\begin{proposition}\label{p:tree}
		$\bigcup_{j=0}^iA_j$ is a tree for each integer $i\ge0$.
	\end{proposition}
	\begin{corollary}\label{c:arcab}
		For every $a,b\in \bigcup_{j=0}^{\infty}A_j$ there is an integer $n\ge0$ such that $ab\subset\bigcup_{j=0}^{n}A_j$.
	\end{corollary}
	
	\begin{proposition}\label{prop:arcab}
		For every $a,b\in D\setminus E_D$ there is an integer $n\ge0$ such that $ab\subset\bigcup_{j=0}^{n}A_j$.
	\end{proposition}
	\begin{proof} Since $a,b\in D\setminus E_D$, the arc $ab$ can be extended from both ends to an arc $a^{\prime}b^{\prime}\subset D$ so that $a^{\prime}a\cap ab=\ofc{a}$ and $ab\cap bb^{\prime}=\ofc{b}$. Let $D_a$ and $D_b$ be dendrites contained in $D\setminus ab$ such that $a^{\prime}\in\operatorname{int}(D_a)$ and $b^{\prime}\in\operatorname{int}(D_b)$. Observe that each point of $ab$ separates $D$ between $D_a$ and $D_b$. Since $\mathcal{S}$ is dense in $D$,
		there are positive integers $n_a$ and $n_b$ such that $s_{n_a}\in D_a$ and $s_{n_b}\in D_b$. Clearly, $ab\subset s_{n_a}s_{n_b}$. Set $n=\max(n_a,n_b)$. Condition \ref{prop:Ai} (\ref{prop:Ai:si}) implies that both $s_{n_a}$ and $s_{n_b}$ belong to $\bigcup_{j=0}^nA_j$. So $s_{n_a}s_{n_b}\subset\bigcup_{j=0}^nA_j$  since $\bigcup_{j=0}^nA_j$ is a tree. Consequently, $ab\subset\bigcup_{j=0}^nA_j$.
	\end{proof}
	
	\begin{corollary}
		$D\setminus E_D\subset\bigcup_{j=0}^\infty A_j$.
	\end{corollary}

\begin{corollary}\label{c:arcinopensubsD}
	For each  nonempty open set $U\subset D$ there is a nonnegative integer $i$ such that $U\cap A_i$ contains a nondegenerate arc.
	\end{corollary}

	For every arc  $L\subset D\setminus E_D$, let $\nu\of{L}$ denote the least nonnegative integer  such that $L\subset\bigcup_{j=0}^{\nu\of{L}}A_j$.
	
	\begin{proposition}\label{p:supdiam}
		Let $d_i$ be the supremum of diameters of arcs contained in $D\setminus\bigcup_{j=0}^iA_j$. Then $\lim_{i\to\infty}d_i=0$.
	\end{proposition}
	\begin{proof}
		Clearly, $d_i\le d_j$ for all integers $i$ and $j$ such that $0\le j\le i$.
		
		Suppose the proposition is false. Then there is a positive number $\epsilon$ such that $d_i>\epsilon$ for all $i=0,1,\dots$. It follows that, for each $i$ there is an arc $J$ contained in $D\setminus\bigcup_{j=0}^iA_j$ such that $\operatorname{diam}(J)>\epsilon$. Let $L$ be a subarc of $J$ such that $L$ is contained in the interior of $J$, but $\operatorname{diam}(L)$ is still greater than $\epsilon$. Obviously, $L\subset J\setminus E_D$. So the following statement is true.
		\begin{claim*}
			For each integer $i\ge0$, there is an arc $L\subset D\setminus\of{E_D\cup \bigcup_{j=0}^iA_j}$ such that $\operatorname{diam}(L)>\epsilon$.
		\end{claim*}
		Let $L_0\subset D\setminus E_D$ be an arc with $\operatorname{diam}(L_0)>\epsilon$.  Use the claim with $i=\nu\of{L_0}$ to get $L_1$ contained in $D\setminus \of{E_D\cup\bigcup_{j=0}^{\nu\of{L_0}}A_j}$ such that $\operatorname{diam}(L_1)>\epsilon$. Continue using the claim repeatedly to  obtain a sequence of arcs $L_1,L_2,L_3,\dots$ such that, for each positive integer $k$, $L_k\subset D\setminus \of{E_D\cup\bigcup_{j=0}^{\nu\of{L_{k-1}}}A_j}$ and $\operatorname{diam}(L_k)>\epsilon$. Observe that the arcs $L_0,L_1,L_2,L_3,\dots$ are mutually disjoint and each of them has diameter greater than $\epsilon$ which is impossible in a dendrite. This contradiction completes the proof of the proposition.
	\end{proof}
	
	For each positive integer $i$, let $l(i)$ be the least nonnegative integer such that $p_i\in A_{l(i)}$. Clearly, $i>l(i)$.
	
	For each nonnegative integer $n$ and each positive integer $i$, let $\mu(n,i)$ denote the set of those integers $j$ such that $1\le j\le i$ and $l(j)=n$. Clearly, $\mu(n,i)=\emptyset$ if $i\le n$. Additionally, set $\mu(n,0)=\emptyset$.
	
	We say that a nonnegative integer $n$ precedes $i$, and write $n\prec i$, if $l^k(i)=n$ for some positive integer $k$. If $n$ does not precede $i$ we write $n\not\prec i$.
	
	Observe that $0\prec i$ for all positive integers $i$.	
	
	\begin{proposition}\label{p:notprec}
		There are no positive integers $i$ and $j$ such that $l\of{j}\prec i\prec j$.
		In particular, if $l\of{i}=l\of{j}$ then neither $i$ precedes $j$ nor $j$ precedes $i$.
	\end{proposition}
	
	For each nonnegative integer $i$, let $C_i$ denote the component of $D\setminus\ofc{p_i}$ containing $A_i\setminus\ofc{p_i}$.
	
	\begin{proposition}\label{p:Ci}
		The following statements are true for each positive integer $i$.
		\begin{enumerate}
			\item $C_i$ is an open path connected set.
			\item\label{p:Ci:pi} $\operatorname{cl}\of{C_i}=C_i\cup\ofc{p_i}$.
			\item $C_i\cap\bigcup_{j=0}^{i-1} A_j=\emptyset$.
			\item Let $j$ be an integer greater than $i$. Then the following three statements are equivalent.
			\begin{itemize}
				\item $A_j\cap C_i\ne\emptyset$.
				\item $A_j\subset C_i$.
				\item $i\prec j$.
			\end{itemize}
		\end{enumerate}
	\end{proposition}
	
	\begin{theorem}[Whyburn]\label{whyburn}
		$D=E_D\cup\bigcup_{i=0}^{\infty}A_i$ and $\lim_{i\to\infty}\operatorname{diam}\of{C_i}=0$.
	\end{theorem}
	\begin{proof}
		The theorem follows from Propositions \ref{p:supdiam} and \ref{p:Ci}.
	\end{proof}
	
	\begin{proposition}\label{p:h*}
		Let $h_0:\bigcup_{j=0}^\infty A_j\to\ofb{0,1}$ such that $\operatorname{diam}\of{h_0\of{A_i}}\le 2^{-i}$ and $h_0$ is continuous on $\bigcup_{j=0}^{i}A_j$ for all nonnegative integer $i$.
		Then there is a unique extension of $h_0$ to a continuous mapping $h:D\to\ofb{0,1}$.
	\end{proposition}
	\begin{proof}
		\begin{claim*}\label{p:h*:cl}
			For each $x\in D$ and each nonnegative integer $i$ there is a continuum $K_i\of{x}\subset B_D\of{x,2^{-i}}$ containing $x$ in its interior such that $\ofa{h_0\of{a}-h_0\of{b}}\le 2^{-i}$ for all $a,b\in K_i\of{x}\cap\bigcup_{j=0}^\infty A_j$.
		\end{claim*}
		\begin{proof}[Proof of the claim]\renewcommand{\qedsymbol}{} If $x\notin\bigcup_{j=0}^{i+1}A_j$, set $T=\emptyset$. Otherwise, let $T\subset \bigcup_{j=0}^{i+1}A_j$ be a tree containing $x$ in its interior with respect to $\bigcup_{j=0}^{i+1}A_j$, and such that $\operatorname{diam}\of{h_0\of{T}}\le2^{-(i+1)}$. Clearly, $Z=\operatorname{cl}\of{\bigcup_{j=0}^{i+1}A_j\setminus T}$ is a compact set not containing $x$. Let $K_i\of{x}\subset B_D\of{x,2^{-i}}\setminus Z$ be a continuum such that $x\in\operatorname{int}\of{K_i\of{x}}$.
			Take any two points $a,b\in K_i\of{x}\cap \bigcup_{j=0}^\infty A_j$. To prove the claim it remains to prove that $\ofa{h_0\of{a}-h_0\of{b}}\le 2^{-i}$.
			
			There is an integer $k>i+1$ such that $ab\subset \bigcup_{j=0}^{k}A_j$, see Corollary \ref{c:arcab}. Since $ab\subset K_i\of{x}\subset D\setminus Z$, we get the result that $ab\subset T\cup \bigcup_{j=i+2}^{k}A_j$. Set $L_{i+1}=ab\cap T$, $L_{i+2}=ab\cap A_{i+2}$, $L_{i+3}=ab\cap A_{i+3}$, \dots, $L_{k}=ab\cap A_{k}$.
			Observe that $\operatorname{diam}\of{h_0\of{L_{j}}}\le2^{-j}$ for all $j=i+1,\dots,k$. Thus
			\begin{equation*}\label{p:h*:cl:eq1}\tag{$*$}
				\sum_{j=i+1}^{k}\operatorname{diam}\of{h_0\of{L_{j}}}\le\sum_{j=i+1}^{k}2^{-j}<\sum_{j=i+1}^{\infty}2^{-j}=2^{-i}
			\end{equation*}
			
			Clearly, $\bigcup_{j=i+1}^{k}L_j=ab$.
			Let $M$ be a subset of $\ofc{i+1,i+2,\dots,k}$ minimal with respect to the property $\bigcup_{j\in M}L_j=ab$. Let $m$ denote the number of elements of $M$. Since the intersection of an arc with a continuum, both contained in a dendrite, is either the empty set, or a point, or a nondegenerate arc, we infer that $L_j$ is a nondegenerate arc for each $j\in M$. Since $\bigcup_{j=j+1}^{k}L_j=ab$ is connected, there is a one-to-one function of $\ofc{1,\dots,m}$ onto $M$ such that $a\in L_{\sigma\of{1}}$ and $L_{\sigma\of{n}}\cap\of{\bigcup_{j=1}^{n-1}L_{\sigma\of{j}}}\ne\emptyset$ for all $n=2,\dots,m$. It follows from the minimality of $M$ that $b\in L_{\sigma\of{m}}$ and
			$L_{\sigma\of{j}}\cap L_{\sigma\of{n}}\ne\emptyset$ if and only if $\ofa{n-j}\le 1$ for all $j,n=1,\dots,m$.
			Consequently, $\ofa{h_0\of{a}-h_0\of{b}}\le \sum_{j=1}^m\operatorname{diam}\of{h_0\of{L_{\sigma\of{j}}}}\le\sum_{j=i+1}^{k}\operatorname{diam}\of{h_0\of{L_{j}}}$. Thus, it follows from \eqref{p:h*:cl:eq1} that $\ofa{h_0\of{a}-h_0\of{b}}<2^{-i}$ and the Claim is true.
		\end{proof}
		
		For an arbitrary point $x\in D$ and an arbitrary nonnegative integer $i$, let $K_i\of{x}$ be the continuum defined in the Claim. Observe that $K^{\prime}_i\of{x}= \bigcap_{j=0}^iK_j\of{x}$ is a continuum containing $x$ in its interior. So, we may replace $K_i\of{x}$ in the claim by $K^{\prime}_i\of{x}$ and have the additional property that $K_{i+1}\of{x}\subset K_i\of{x}$ for each nonnegative integer $i$.
		
		$K_i\of{x}\cap\bigcup_{j=0}^\infty A_j\ne\emptyset$ because $K_i\of{x}$ has nonempty interior and $\bigcup_{j=0}^\infty A_j$ is dense in $D$. So, $H_i\of{x}=h_0\of{K_i\of{x}\cap\bigcup_{j=0}^\infty A_j}$ is not empty. It follows from the choice of $K_i\of{x}$ that $H_{i+1}\of{x}\subset H_i\of{x}$ and $\operatorname{diam}\of{H_i\of{x}}\le2^{-i}$.
		Consequently, $\operatorname{cl}\of{H_i\of{x}}\subset\ofb{0,1}$ is a closed nonempty set, $\operatorname{cl}\of{H_{i+1}\of{x}}\subset \operatorname{cl}\of{H_i\of{x}}$ and $\operatorname{diam}\of{\operatorname{cl}\of{H_i\of{x}}}\le2^{-i}$ for all nonnegative $i$. It follows that $ \bigcap_{j=0}^{\infty} \operatorname{cl}\of{H_j\of{x}}$ is a single point. We denote this point by $h\of{x}$. Clearly, $h\of{x}\in \operatorname{cl}\of{H_j\of{x}}$ for all nonnegative integers $j$.
		
		We will show that $h$ is continuous. Take an arbitrary point $x\in D$ and a positive number $\epsilon$. We will show that there is an open neighborhood $U$ of $x$ in $D$ such that $\ofa{h\of{z}-h\of{x}}<\epsilon$ for each $z\in U$. Let $i$ be a nonnegative integer such that $2^{-i}<\epsilon$. Set $U=\operatorname{int}\of{K_i\of{x}}$ and take an arbitrary point $z\in U$. There is an integer $n$ such that $B_D\of{z,2^{-n}}\subset U=\operatorname{int}\of{K_i\of{x}}\subset K_i\of{x}$. Hence $K_n\of{z}\subset K_i\of{x}$. It follows that $H_n\of{z}= h_0\of{K_n\of{z}\cap\bigcup_{j=0}^\infty A_j} \subset h_0\of{K_i\of{x}\cap\bigcup_{j=0}^\infty A_j}=H_i\of{x}$.
		So, $h\of{z}\in \operatorname{cl}\of{H_n\of{z}} \subset \operatorname{cl}\of{H_i\of{x}}$. Since $\operatorname{diam}\of{\operatorname{cl}\of{H_i\of{x}}}\le2^{-i}$ and both $h\of{z}$ and $h\of{x}$ belong to $\operatorname{cl}\of{H_i\of{x}}$, we have the result $\ofa{h\of{z}-h\of{x}}\le 2^{-i}<\epsilon$. Hence, $h$ is continuous.
		
		Finally, we must observe that $h$ is an extension of $h_0$. Suppose that $x\in \bigcup_{j=0}^\infty A_j$. Then $x\in K_i\of{x}\cap\bigcup_{j=0}^\infty A_j$ for each nonnegative integer $i$. It follows that $h_0\of{x}\in H_i\of{x}$ for all $i\ge 0$.  Consequently, $\bigcap_{i=0}^{\infty} \operatorname{cl}\of{H_i\of{x}}=\ofc{h_0\of{x}}$ and, therefore, $h\of{x}=h_0\of{x}$. The extension is unique since it is continuous and 
		$\bigcup_{j=0}^\infty A_j$ is dense in $D$.
	\end{proof}
	
	\section{Factorization Lemma and the Small Folds Property}\label{factor}
	\begin{equation*}
		\begin{tikzcd}[baseline=\the\dimexpr\fontdimen22\textfont2\relax]
			&  & X \ar[ld,"g"]& X \ar[l,dashed,"h\circ g"]\ar[ld,"g"] & X \ar[l,dashed,"h\circ g"]\ar[ld,"g"]& X \ar[l,dashed,"h\circ g"]\ar[ld,"g"]& X \ar[l,dashed,"h\circ g"]\ar[ld,"g"]&\ar[l, dotted]\\
			& Y & Y \ar[u,"h"] \ar[l,dashed,"g\circ h"] & Y \ar[u,"h"]\ar[l,dashed,"g\circ h"] & Y \ar[u,"h"]\ar[l,dashed,"g\circ h"]& Y \ar[u,"h"]\ar[l,dashed,"g\circ h"] &\ar[l, dotted]\\
		\end{tikzcd}
	\end{equation*}
	\begin{proposition}\label{p:hg-gh}
		Let $X$ and $Y$ be two compact spaces, and let $g:X\to Y$ and $h:Y\to X$ be two continuous mappings. Then $\varprojlim\of{X,h\circ g}$ and $\varprojlim\of{Y,g\circ h}$ are homeomorphic. Moreover, the following statements are true.
\begin{enumerate}
  \item\label{p:hg-gh:tr} Suppose $g$ is a surjection and $h\circ g$ is transitive on $X$.  Then $g\circ h$ is transitive on $Y$.
  \item\label{p:hg-gh:tm} Suppose $g$ is a surjection and $h\circ g$ is topologically mixing on $X$.  Then $g\circ h$ is topologically mixing on $Y$.
  \item\label{p:hg-gh:nnk2p} Suppose $g$ is a surjection, $\operatorname{int}_X\of{h(U)}\ne\emptyset$ for each nonempty open set $U\subset Y$, and $h\circ g$ is topologically exact on $X$.  Then $g\circ h$ is topologically exact on $Y$.
\end{enumerate}

	\end{proposition}
	\begin{proof}
		Consider the sequence $\of{Z_i}_{i=1}^\infty$ where $Z_i=X$ for even $i$ and $Z_i=Y$ odd $i$. Let $f_i:Z_{i+1}\to Z_i$ be $h$ if $i$ is even and $g$ if $i$ is odd.
		Observe that restricting all threads $\of{z_i}_{i=0}^\infty\in\varprojlim\of{Z_i,f_i}$ to even terms results in all threads belonging to $\varprojlim\of{X,h\circ g}$. Since such restriction is a homeomorphism between the corresponding inverse limits, see \cite[Corollary 2.5.11]{engelking}, $\varprojlim\of{X,h\circ g}$ and $\varprojlim\of{Z_i,f_i}$ are homeomorphic. Similarly, $\varprojlim\of{Y,g\circ y}$ and $\varprojlim\of{Z_i,f_i}$ are homeomorphic, since restricting $\of{z_i}_{i=0}^\infty$ to odd terms results in all threads belonging to $\varprojlim\of{Y,g\circ h}$. Hence, $\varprojlim\of{X,h\circ g}$ and $\varprojlim\of{Y,g\circ h}$ are homeomorphic.
		
Suppose that assumptions of the statement (\ref{p:hg-gh:tr}) are satisfied. Then there is $x\in X$ such that $\of{\of{h\circ g}^i\of{x}}_{i=1}^\infty$ is dense in $X$. Observe that
		$g\of{\of{h\circ g}^i\of{x}}=\of{g\circ h}^i\of{g\of{x}}$ for each positive integer $i$. Since $g$ is a surjection, the image of a dense set in $X$ is dense in $Y$. Consequently, $\of{g\of{\of{h\circ g}^i\of{x}}}_{i=1}^\infty=\of{\of{g\circ h}^i\of{g\of{x}}}_{i=1}^\infty$ is dense in $Y$. So, the orbit of $g\of{x}$ under $g\circ h$ is dense in $Y$. Thus, $g\circ h$ is transitive on $Y$ and the statement (\ref{p:hg-gh:tr}) is true.

Now, suppose that assumptions of the statement (\ref{p:hg-gh:tm}) are satisfied. Let $U$ and $V$ be arbitrary open nonempty subsets of $Y$. Clearly, $g{}^{-1}\of{U}$ and $g{}^{-1}\of{V}$ are open nonempty subsets of $X$. Also, $g\of{g{}^{-1}\of{U}}=U$ and $g\of{g{}^{-1}\of{V}}=V$. Since $h\circ g$ is topologically mixing, there exists a number $N$ such that $\of{h\circ g}^i\of{g{}^{-1}\of{U}}\cap g{}^{-1}\of{V}\ne\emptyset$ for all $i>N$. So,
\begin{equation*}
  g\of{\of{h\circ g}^i\of{g{}^{-1}\of{U}}\cap g{}^{-1}\of{V}} \ne\emptyset \quad \text{for all $i>N$.}
\end{equation*}
Since $g\of{A\cap B}\subset g\of{A}\cap g\of{B}$ for all $A,B\subset X$, we infer that
\begin{equation*}
  g\of{\of{h\circ g}^i\of{g{}^{-1}\of{U}}}\cap g\of{g{}^{-1}\of{V}} \ne\emptyset \quad \text{for all $i>N$.}
\end{equation*}
Since $g\of{\of{h\circ g}^i\of{g{}^{-1}\of{U}}}=\of{g\circ h}^i\of{g\of{g{}^{-1}\of{U}}}=\of{g\circ h}^i\of{U}$ and $g\of{g{}^{-1}\of{V}} = V$, we get the result that
$\of{g\circ h}^i\of{U}\cap V \ne\emptyset$ for all $i>N$. Hence, the statement (\ref{p:hg-gh:tm}) is true.

Finally, suppose that assumptions of the statement (\ref{p:hg-gh:nnk2p}) are satisfied. Let $U$ be an arbitrary nonempty open subset of $Y$. Then $V=\operatorname{int}_X\of{h(U)}$ is a nonempty open subset of $X$. Since  $h\circ g$ is topologically exact on $X$, there is a positive integer $i$ such that $\of{h\circ g}^i\of{V}=X$. It follows that $g\circ\of{h\circ g}^i\of{V}=g\of{X}=Y$ since $g$ is a surjection. Since $V\subset h\of{U}$, we infer that  $\of{g\circ h}^{i+1}\of{U}=g\circ\of{h\circ g}^i\of{h\of{U}}\supset g\circ\of{h\circ g}^i\of{V}=Y$.
Consequently,   $\of{g\circ h}^{i+1}\of{U}=Y$ and $g\circ h$ is topologically exact on $Y$.
\end{proof}
	
	\begin{proposition}\label{p:weakconf}
		Let $f$ be a continuous real function defined on an interval $I$. Suppose $a,b\in f\of{I}$ and $a\ne b$.
		Then there are points $c,d\in I$ such that $f(c)=a$, $f(d)=b$ and $f(t)\in\of{a,b}$ for each $t\in\of{c,d}$.
	\end{proposition}
	\begin{proof}
		Let $c_o,d_0\in I$ be such that $f\of{c_0}=a$ and $f\of{d_0}=b$. Let $d$ be the first point in the oriented interval $\ofb{c_0,d_0}$ such that $f\of{d}=b$. Finally, let $c$ be the last point in the oriented interval $\ofb{c_0,d}$ such that $f\of{c}=a$.
	\end{proof}
	
	\begin{definition}[see {\cite[p 1166]{minc-transue}}]
		Let $f:\ofb{0, 1}\to\ofb{0, 1}$ be a continuous function. Let $a$ and $b$ be
		two points of the interval $\ofb{0, 1}$, and let $\delta$ be a positive number. We say
		that \emph{$f$ is $\delta$-crooked between $a$ and $b$} if, for every two points $c,d\in\ofb{0, 1}$
		such that $f(c) = a$ and $f(d) = b$ , there is a point $c^{\prime}$ between $c$ and $d$
		and there is a point $d^{\prime}$ between $c^{\prime}$ and $d$ such that $\ofa{b - f(c^{\prime})} \le\delta$ and
		$\ofa{a - f(d^{\prime})} \le\delta$. We say that \emph{$f$ is $\delta$-crooked} if it is $\delta$-crooked between
		every pair of points.
	\end{definition}
	
	\begin{definition}\label{def:smallfolds}
		Let $f:\ofb{0, 1}\to\ofb{0, 1}$ be a continuous function. We say that $f$ has \emph{the small folds property} if, for every positive number $\lambda<1$, there exist positive numbers $\beta<\lambda$ and $\xi<\beta/4$ satisfying the following condition
		\begin{equation*}
			\text{for every $a$ and $b$ such that} \ofa{a - b} < \beta \text{, $f$ is $\xi$-crooked between $a$ and $b$.}
		\end{equation*}
	\end{definition}
	
	\begin{lemma}[Factorization Lemma]\label{l:factf}
		Let $f:\ofb{0, 1}\to\ofb{0, 1}$ be a continuous surjection with the small folds property and let $D$ be a dendrite.
		Then there are continuous surjections $g:\ofb{0,1}\to D$ and $h:D\to\ofb{0,1}$ such that $h\circ g=f$ and $\operatorname{int}_{\ofb{0,1}}\of{h(U)}\ne\emptyset$ for each nonempty open set $U\subset D$.
	\end{lemma}
	\begin{proof} We will assume here that $D$ is not a tree. The proof in case where $D$ is a tree is similar, but much simpler. We include a short sketch of the proof in this case at the end of our argument.
		
		Let $A_0=p_0q_0, A_1=p_1q_1, A_2=p_2q_2,\dots$, $l(i)$, $\mu(n,i)$ and $\prec$  be as in Section \ref{s:dendrite}.	
		
		Let $r_0=0$, $s_0=1$, and let $\tau_0$ be a homeomorphism of $\ofb{r_0,s_0}$ onto $A_0$ such that $\tau_0\of{r_0}=p_0$ and $\tau_0\of{s_0}=q_0$.
		Additionally, let $u_0,v_0\in\ofb{0,1}$ be such that $u_0<v_0$ and the interval $\ofb{u_0,v_0}$ is minimal with respect to the property $f\of{\ofb{u_0,v_0}}=\ofb{0,1}=\ofb{r_0,s_0}$.
		
		We will construct sequences $\of{r_i}_{i=1}^\infty$, $\of{s_i}_{i=1}^\infty$, $\of{\tau_i}_{i=1}^\infty$, $\of{u_i}_{i=1}^\infty$,  and $\of{v_i}_{i=1}^\infty$ satisfying the following conditions for all positive integers $i$.
				
		\begin{list}{$(\arabic{lcount})_i$ }{\usecounter{lcount}}
			\item\label{l:factf:cond1} $0\le r_i<s_i\le r_i+2^{-i}$.
			\item \label{l:factf:tau}$\tau_i$ is a homeomorphism of $\ofb{r_i,s_i}$ onto $A_i$ such that $\tau_i\of{r_i}=p_i$ and $\tau_i\of{s_i}=q_i$.
			\item \label{l:factf:tauinverse} $r_i=\tau_{l\of{i}}{}^{-1}\of{p_i}$.
			\item $u_i<v_i$.
			\item \label{l:factf:fui=fvi}$f\of{u_i}=f\of{v_i}=r_i$.
			\item $f\of{t}> r_i$ for $t\in\of{u_i,v_i}$.
			\item\label{l:factf:si} $s_i=\max\of{f\ofb{u_i,v_i}}$.
			\item\label{l:factf:diamtauli} $\operatorname{diam}\of{\tau_{l(i)}\of{f\of{\ofb{u_i,v_i}}}}<2^{-i}$ where $\operatorname{diam}(*)$ is the diameter in $D$.
			\item\label{l:factf:intesofunvnanduivi} Suppose $n$ is integer such that $0\le n< i$. Then the following three statements are equivalent.
			\begin{itemize}
				\item $\ofb{u_i,v_i}\cap\ofb{u_n,v_n}\ne\emptyset$.
				\item $\ofb{u_i,v_i}\subset\of{u_n,v_n}$.
				\item $n\prec i$.
			\end{itemize}
			\item\label{l:factf:xinf(I)} Suppose $n$ is an integer such that $0\le n\le i$. Suppose also $x\in\of{r_n,s_n}$. Then there is an interval   $I\subset \of{u_n,v_n}\setminus\bigcup_{j\in\mu(n,i)}\ofb{u_j,v_j}$ such that $x\in\operatorname{int}\of{f\of{I}}$.
		\end{list}
		Observe that $(\ref{l:factf:xinf(I)})_0$ is satisfied.
		
		Let $i$ be a positive integer. Suppose $\of{r_j}_{j=0}^{i-1}$, $\of{s_j}_{j=0}^{i-1}$, $\of{\tau_j}_{j=0}^{i-1}$, $\of{u_j}_{j=1}^{i-1}$,  and $\of{v_j}_{j=1}^{i-1}$ satisfying  the above conditions have been constructed. We will now construct $r_i$, $s_i$, $\tau_i$, $u_i$ and $v_i$.
		
		Set $r_i=\tau_{l(i)}{}^{-1}\of{p_i}$. Using $(\ref{l:factf:xinf(I)})_{i-1}$  with $n=l(i)$ and $x=r_i$ we get an interval $I\subset\of{u_{l(i)},v_{l(i)}}\setminus\bigcup_{j\in\mu\of{l(i),i-1}}\ofb{u_j,v_j}$ such that $r_i\in\operatorname{int}\of{f\of{I}}$.
		Let $\lambda$ be a positive number satisfying the following conditions.
		\begin{list}{($\lambda$-\arabic{lcount})}{\usecounter{lcount}}
			\item $\lambda< 2^{-i}$,
			\item $\operatorname{diam}\of{\tau_{l(i)}\of{\ofb{r_i-\lambda,r_i+\lambda}}}<2^{-i}$, and
			\item $\ofb{r_i-\lambda,r_i+\lambda}\subset f\of{I}$.
		\end{list}
		Let $\beta<\lambda$ and $\xi<\beta/4$ be positive numbers as in Definition \ref{def:smallfolds}.
		Set $a=r_i-2\xi$ and $b=r_i+2\xi$. Clearly, $\ofb{a,b}\subset\of{r_i-\lambda,r_i+\lambda}\subset f\of{I}$.
		Using Proposition \ref{p:weakconf} we get points $c,d\in I$ such $f(c)=a$, $f(d)=b$ and $f(t)\in\of{a,b}$ for each $t$ between $c$ and $d$. Since $b-a=4\xi<\beta$ and $f$ is $\xi$-crooked between $a$ and $b$,
		there is a point $c^{\prime}$ between $c$ and $d$
		and there is a point $d^{\prime}$ between $c^{\prime}$ and $d$ such that $\ofa{b - f(c^{\prime})} \le\xi$ and
		$\ofa{a - f(d^{\prime})} \le\xi$. It follows that
		\begin{equation*}
			r_i+\xi\le f\of{c^{\prime}}<r_i+2\xi=b \quad\text{ and }\quad a=r_i-2\xi<f\of{d^{\prime}}\le r_i-\xi.
		\end{equation*}
		
		We will now consider the cases $c<d$ and $d<c$ to define $u_i$, $v_i$ and an interval $J\subset \of{u_{l(i)},v_{l(i)}}\setminus\bigcup_{j\in\mu\of{l(i),i}}\ofb{u_j,v_j}$ such that
		\begin{equation*}\label{l:factf:J}
			f\of{\ofb{u_i,v_i}}\subset\operatorname{int}\of{f\of{J}}.\tag{$*$}
		\end{equation*}
		
		Case: $c<d$. In this case $c<c^{\prime}<d^{\prime}<d$. Let $u_i$ be the greatest number in the interval $\ofb{c,c^{\prime}}$ such that $f\of{u_i}=r_i$, and let $v_i$ be the least number in the interval $\ofb{c^{\prime},d^{\prime}}$ such that $f\of{v_i}=r_i$. Also, set $J=\ofb{d^{\prime},d}$.
		
		Case: $d<c$. In this case $d<d^{\prime}<c^{\prime}<c$. Let $u_i$ be the greatest number in the interval $\ofb{d^{\prime},c^{\prime}}$ such that $f\of{u_i}=r_i$, and let $v_i$ be the least number in the interval $\ofb{c^{\prime},c}$ such that $f\of{v_i}=r_i$. Also, set $J=\ofb{d,d^{\prime}}$.
		
		Observe that \thetag{\ref{l:factf:J}} is satisfied in both cases. To conclude the construction we set $s_i=\max\of{f\ofb{u_i,v_i}}$ as required in condition $(\ref{l:factf:si})_i$. It is easy to check that conditions $(\ref{l:factf:cond1})_i$-$(\ref{l:factf:intesofunvnanduivi})_i$ are true.
		
		Proof of $(\ref{l:factf:xinf(I)})_i$. If $n=i$, then $\mu\of{n,i}=\emptyset$. So, $\of{u_n,v_n}\setminus\bigcup_{j\in\mu(n,i)}\ofb{u_j,v_j}=\of{u_n,v_n}$ and $(\ref{l:factf:xinf(I)})_i$ follows from $(\ref{l:factf:fui=fvi})_i$ and $(\ref{l:factf:si})_i$. So we may assume that $n<i$. Using $(\ref{l:factf:xinf(I)})_{i-1}$ for
		$x\in\of{u_n,v_n}$ we infer that there is an interval $I_{i-1}\subset \of{u_n,v_n}\setminus\bigcup_{j\in\mu(n,i-1)}\ofb{u_j,v_j}$ such that $x\in\operatorname{int}\of{f\of{I_{i-1}}}$. If $n\ne l(i)$, then $i\notin\mu\of{n,i}$, $\mu\of{n,i}=\mu\of{n,i-1}$ and $(\ref{l:factf:xinf(I)})_i$ is satisfied by letting $I=I_{i-1}$. So, we may assume that $n=l(i)$. To finish the proof of $(\ref{l:factf:xinf(I)})_i$ we will consider the following two cases $x\notin f\of{\ofb{u_i,v_i}}$ and $x\in f\of{\ofb{u_i,v_i}}$.
		
		Case $x\notin f\of{\ofb{u_i,v_i}}$. In this case there is an interval $L\subset f\of{I_{i-1}}$ such that $x\in\operatorname{int}\of{L}$ and $L\cap f\of{\ofb{u_i,v_i}}=\emptyset$.
		It follows from \ref{p:weakconf} that there is an interval $I\subset I_{i-1}$ such that $f(I)=L$. Observe that this choice of $I$ satisfies condition $(\ref{l:factf:xinf(I)})_i$.
		
		Case $x\in f\of{\ofb{u_i,v_i}}$. In this case set $I=J$ and observe that $(\ref{l:factf:xinf(I)})_i$ follows from \thetag{\ref{l:factf:J}}.
		
		The construction of $\of{r_i}_{i=1}^\infty$, $\of{s_i}_{i=1}^\infty$, $\of{\tau_i}_{i=1}^\infty$, $\of{u_i}_{i=1}^\infty$,  and $\of{v_i}_{i=1}^\infty$ satisfying $(\ref{l:factf:cond1})_i$-$(\ref{l:factf:xinf(I)})_i$ is now complete.
		
		Let $h_0$ be a real function of $\bigcup_{j=0}^{\infty}A_j$ defined by $h_0\of{x}=\tau_i{}^{-1}\of{x}$ for $x\in A_i$ for every nonnegative integer $i$. Observe that conditions $(\ref{l:factf:cond1})_i$-$(\ref{l:factf:si})_i$ guarantee that $h_0$ is a well-defined function onto $\ofb{0,1}$ satisfying the assumptions of Proposition  \ref{p:h*}.
		Thus, there is a unique extension of $h_0$ to a continuous mapping $h:D\to\ofb{0,1}$.
		
 Since  $\tau_i{}^{-1}$ is an embedding of $A_i$ into $\ofb{0,1}$ for each nonnegative integer $i$, it follows from Corollary \ref{c:arcinopensubsD} that  $\operatorname{int}_{\ofb{0,1}}\of{h(U)}\ne\emptyset$ for each nonempty open set $U\subset D$.

		For each nonnegative integer $i$, we will define a function $g_i:\ofb{0,1}\to\bigcup_{j=0}^{i}A_j$ by a recursive formula:
		Set $g_0=\tau_0\circ f$ and, for each positive integer $i$, let $g_i$ be defined by
		$$g_i(t)=\left\{
		\begin{array}{ll}
			g_{i-1}(t), & \hbox{if $t\in\ofb{0,1}\setminus\of{u_i,v_i}$;} \\
			\tau_i\circ f(t), & \hbox{if $t\in\of{u_i,v_i}.$}
		\end{array}
		\right.
		$$
		
		The following claim is an easy consequence of the above definition.
		\begin{claim}\label{l:factf:c0}
			Suppose $n$ and $i$ are integers such that $0\le n<i$. Then $g_i\of{t}=g_n\of{t}$ for each $t\in\ofb{0,1}\setminus\bigcup_{j=n+1}^i\of{u_j,v_j}$.
		\end{claim}
		
		\begin{claim}\label{l:factf:c}
			Let $i$ be a nonnegative integer. Then the following properties are true.
			\begin{list}{$(\rm{P-}\arabic{lcount})_i$ }{\usecounter{lcount}}
				\item \label{l:factf:c:cont}$g_i$ is a continuous surjection onto $\bigcup_{j=0}^iA_j$.
				\item \label{l:factf:c:h}$h\circ g_i=f$.
				\item \label{l:factf:c:n}Suppose $n$ is an integer such that $0\le n\le i$ then
				\begin{list}{$(\rm\roman{lcount1})_n$}{\usecounter{lcount1}}
					\item\label{l:factf:c:n:t} $g_i\of{t}=g_n\of{t}=\tau_n\circ f\of{t}$ for $t\in\ofb{u_n,v_n}\setminus\bigcup_{j\in\mu\of{n,i}}\of{u_j,v_j}$,
					\item\label{l:factf:c:n:A} $g_i\of{\ofb{u_n,v_n}\setminus\bigcup_{j\in\mu\of{n,i}}\of{u_j,v_j}}=A_n$, and
					\item\label{l:factf:c:n:C} $g_i\of{\of{u_n,v_n}}\subset C_n$.
				\end{list}
			\end{list}
		\end{claim}
		
		\begin{proof}[Proof of \ref{l:factf:c}]\renewcommand{\qedsymbol}{}
			
			We will prove the claim by induction with respect to $i$. Observe that $(\rm{P-}\ref{l:factf:c:cont})_0$ -- $(\rm{P-}\ref{l:factf:c:n})_0$ are true. Suppose that $i$ is a positive integer such that $(\rm{P-}\ref{l:factf:c:cont})_{i-1}$ -- $(\rm{P-}\ref{l:factf:c:n})_{i-1}$ are satisfied. We will prove $(\rm{P-}\ref{l:factf:c:cont})_i$ -- $(\rm{P-}\ref{l:factf:c:n})_i$.
			
			Clearly, $l\of{i}<i$. If $j\in\mu\of{l\of{i},i-1}$ then $j\not\prec i$ by Proposition \ref{p:notprec}. So, it follows from $(\ref{l:factf:intesofunvnanduivi})_i$ used with $n=j$ that
			$\ofb{u_i,v_i}\cap\bigcup_{j\in\mu\of{l\of{i},i-1}}\ofb{u_j,v_j}=\emptyset$. Using $(\ref{l:factf:intesofunvnanduivi})_i$ again, this time with $n=l\of{i}$ we get the result that
			$\ofb{u_i,v_i}\subset\of{u_{l\of{i}},v_{l\of{i}}}$. It follows from $(\ref{l:factf:tau})_{l\of{i}}$, $(\ref{l:factf:tauinverse})_i$, $(\ref{l:factf:fui=fvi})_i$ and  $(\rm{P-}\ref{l:factf:c:n})_{i-1}$ $(\rm\romannumeral 0\ref{l:factf:c:n:t}\relax)_{l\of{i}}$ that $g_{i-1}\of{u_i}=g_{i-1}\of{v_i}=p_i$.
			So, $g_{i-1}$ restricted to $\ofb{0,1}\setminus\of{u_i,v_i}$ and $\tau_i\circ f$ defined on $\ofb{u_i,v_i}$ are two continuous functions agreeing on the intersection of their (compact) domains. Consequently, $g_i$ which is the union of these two functions is continuous on the interval $\ofb{0,1}$. Also, observe that this definition of $g_i$ guarantees    $(\rm{P-}\ref{l:factf:c:n})_{i}$ $(\rm\romannumeral 0\ref{l:factf:c:n:t}\relax)_{i}$. If $0\le n<i$ then $(\rm{P-}\ref{l:factf:c:n})_{i}$ $(\rm\romannumeral 0\ref{l:factf:c:n:t}\relax)_{n}$  follows automatically from $(\rm{P-}\ref{l:factf:c:n})_{i-1}$ $(\rm\romannumeral 0\ref{l:factf:c:n:t}\relax)_{n}$ because $\mu\of{n,i-1}\subset\mu\of{n,i}$. So, $(\rm{P-}\ref{l:factf:c:n})_{i}$ $(\rm\romannumeral 0\ref{l:factf:c:n:t}\relax)_{n}$ is true for all integers $n$ such that $0\le n\le i$.
			The property $(\rm{P-}\ref{l:factf:c:n})_{i}$ $(\rm\romannumeral 0\ref{l:factf:c:n:A}\relax)_n$ follows from continuity of $g_i$, $(\rm{P-}\ref{l:factf:c:n})_{i}$ $(\rm\romannumeral 0\ref{l:factf:c:n:t}\relax)_{n}$ and $(\ref{l:factf:xinf(I)})_i$.
			
			Proof of $(\rm{P-}\ref{l:factf:c:n})_{i}$ $(\rm\romannumeral 0\ref{l:factf:c:n:C}\relax)_{n}$. Observe that $(\rm{P-}\ref{l:factf:c:n})_{i}$ $(\rm\romannumeral 0\ref{l:factf:c:n:C}\relax)_{i}$ is true since $g_i\of{\of{u_i,v_i}}\subset A_i\setminus\ofc{p_i}\subset C_i$. Hence, it is enough to prove $(\rm{P-}\ref{l:factf:c:n})_{i}$ $(\rm\romannumeral 0\ref{l:factf:c:n:C}\relax)_{n}$ for each nonnegative integer $n<i$. In this case we may use $(\rm{P-}\ref{l:factf:c:n})_{i-1}$ $(\rm\romannumeral 0\ref{l:factf:c:n:C}\relax)_{n}$ to infer that $g_{i-1}\of{\of{u_n,v_n}}\subset C_n$. Suppose $n\not\prec i$. Then $\ofb{u_i,v_i}\cap\ofb{u_n,v_n}=\emptyset$ by $(\ref{l:factf:intesofunvnanduivi})_i$, and $g_i\mid\ofb{u_n,v_n}=g_{i-1}\mid\ofb{u_n,v_n}$. So $g_{i}\of{\of{u_n,v_n}}=g_{i-1}\of{\of{u_n,v_n}}\subset C_n$. Hence, we may assume that $n\prec i$. In such case $\ofb{u_i,v_i}\subset\of{u_n,v_n}$ by $(\ref{l:factf:intesofunvnanduivi})_i$. Consequently, $g_i\of{u_i}=g_{i-1}\of{u_i}=\tau_n\circ f\of{u_i}=p_i$ belongs to $C_i$. Thus, $A_i\subset C_i$ since $p_n\notin A_i$. This implies that $g_i\of{\ofb{u_i,v_i}}\subset A_i\subset C_n$. It follows that $g_i\of{\ofb{u_n,v_n}}\subset C_n$
			since $g_{i}\of{\ofb{u_n,v_n}\setminus\of{u_i,v_i}}=g_{i-1}\of{\ofb{u_n,v_n}\setminus\of{u_i,v_i}}$. This completes the proof of $(\rm{P-}\ref{l:factf:c:n})_{i}$ $(\rm\romannumeral 0\ref{l:factf:c:n:C}\relax)_{n}$, and the proof of $(\rm{P-}\ref{l:factf:c:n})_{i}$ in general.
			
			It follows from $(\rm{P-}\ref{l:factf:c:n})_{i}$ $(\rm\romannumeral 0\ref{l:factf:c:n:A}\relax)_n$ that $g_i\of{\ofb{0,1}}=\bigcup_{j=0}^iA_j$. So, $(\rm{P-}\ref{l:factf:c:cont})_i$ is true since we  have already proven that $g_i$ is continuous.
			
			To show $(\rm{P-}\ref{l:factf:c:h})_i$, recall that $h\mid \bigcup_{j=0}^\infty A_j=h_0$ and $h_0\of{x}=\tau_i{}^{-1}\of{x}$ for all $x\in A_i$. It follows from the definition of $g_i$ that $g_i\of{t}=\tau_i\circ f\of{t}\in A_i$ for all $t\in\of{u_i,v_i}$. So, $h\circ g_i\of{t}=\tau_i{}^{-1}\circ\tau_i\circ f\of{t}=f\of{t}$ for all $t\in\of{u_i,v_i}$.
			Now, $(\rm{P-}\ref{l:factf:c:h})_i$ follows from $(\rm{P-}\ref{l:factf:c:h})_{i-1}$. Hence, the claim is true.
		\end{proof}
		\begin{claim}\label{l:factf:cauchy}
			$\of{g_i}$ is a Cauchy sequence.
		\end{claim}
		\begin{proof}[Proof of \ref{l:factf:cauchy}]\renewcommand{\qedsymbol}{}
			Let $\epsilon$ be an arbitrary positive number. It follows from Theorem \ref{whyburn} that there is an integer $m$ such that $2^{-m}<\epsilon$ and
			$\operatorname{diam}\of{C_j}<\epsilon/2$ for each $j\ge m$. Let $i$ be an arbitrary integer grater than $m$ and let $t$ be an arbitrary element of $\ofb{0,1}$.
			To complete the proof of the claim, we will show that
			\begin{equation*}\label{l:factf:cauchy:?}
				d\of{g_i\of{t},g_m\of{t}}<\epsilon \tag{$*0$}
			\end{equation*}
			If $t\notin\bigcup_{j=m+1}^i\of{u_j,v_j}$ then $g_i\of{t}=g_m\of{t}$ by \ref{l:factf:c0}, and the \eqref{l:factf:cauchy:?} is true. So, we may assume that $t\in\bigcup_{j=m+1}^i\of{u_j,v_j}$. Let $n$ be the least integer such that $m<n\le i$ and $t\in\of{u_n,v_n}$.
			It follows from $(\rm{P-}\ref{l:factf:c:n})_{i}$ $(\rm\romannumeral 0\ref{l:factf:c:n:C}\relax)_n$ that $g_i\of{t}\in C_n$.
			Since $p_n\in\operatorname{cl}\of{C_n}$ by Proposition \ref{p:Ci}(\ref{p:Ci:pi}), we infer that
			\begin{equation*}\label{l:factf:cauchy:*1}
				d\of{g_i\of{t},p_n}\le \operatorname{diam}\of{C_n}<\epsilon/2 \tag{$*1$}
			\end{equation*}
			
			Clearly, $l\of{n}<n\le i$. Since $t\in\of{u_n,v_n}\subset\of{u_{l\of{n}},v_{l\of{n}}}$ by $(\ref{l:factf:intesofunvnanduivi})_n$, the choice of $n$ implies that $l\of{n}\le m$.
			
			Suppose there exists an integer $j$ such that $l\of{n}\le m$ and $t\in\of{u_j,v_j}$. Then, since $m<n$ and $t\in\of{u_n,v_n}\cap\of{u_j,v_j}\cap\of{u_{l\of{n}},v_{l\of{n}}}$, $(\ref{l:factf:intesofunvnanduivi})_j$ and $(\ref{l:factf:intesofunvnanduivi})_n$ imply that $l\of{n}\prec j\prec n$ which contradicts Proposition \ref{p:notprec}. So, $t\notin\bigcup_{j=l(n)+1}^m\of{u_j,v_j}$ and Claim \ref{l:factf:c0} implies
			\begin{equation*}\label{l:factf:cauchy:*2}
				g_{l\of{n}}\of{t}=g_m\of{t} \tag{$*2$}
			\end{equation*}
			
			Using $(\ref{l:factf:tau})_n$, $(\ref{l:factf:tauinverse})_n$ and $(\ref{l:factf:fui=fvi})_n$ we infer that $\tau_{l\of{n}}\of{f\of{u_n}}=p_n$. It follows from $(\rm{P-}\ref{l:factf:c:n})_{l\of{n}}$ $(\rm\romannumeral 0\ref{l:factf:c:n:t}\relax)_{l\of{n}}$ that $g_{l\of{n}}\of{u_n}=\tau_{l\of{n}}\of{f\of{u_n}}=p_n$ and
			$g_{l\of{n}}\of{t}=\tau_{l\of{n}}\of{f\of{t}}$. We now apply $(\ref{l:factf:diamtauli})_n$ to estimate the distance between $g_{l\of{n}}\of{t}$ and $p_n$ in the following way:
			$d\of{p_n,g_{l\of{n}}\of{t}}=d\of{\tau_{l\of{n}}\of{f\of{u_n}},\tau_{l\of{n}}\of{f\of{t}}}\le \operatorname{diam}\of{\tau_{l\of{n}}\of{f\ofb{u_n,v_n}}}<2^{-n}$.
			Since $2^{-n}\le2^{-m-1}<\epsilon$ we get the result
			\begin{equation*}\label{l:factf:cauchy:*3}
				d\of{p_n,g_{l\of{n}}\of{t}}<\epsilon/2 \tag{$*3$}
			\end{equation*}
			
			Combining \eqref{l:factf:cauchy:*1}, \eqref{l:factf:cauchy:*3} and \eqref{l:factf:cauchy:*2} we infer that
			\begin{equation*}\label{l:factf:cauchy:!}
				d\of{g_i\of{t},g_m\of{t}}\le  d\of{g_i\of{t},p_n}+d\of{p_n,g_{l\of{n}}\of{t}}+d\of{g_{l\of{n}}\of{t},g_m\of{t}}<\epsilon/2+\epsilon/2+0
			\end{equation*}
			Hence, \eqref{l:factf:cauchy} is true and the proof of the claim is complete.
		\end{proof}
		
		Let $g=\lim_{i\to\infty}g_i$. Clearly, $g$ is continuous as the limit of a uniformly convergent sequence of continuous functions into a compact space $D$.  Observe that $g\of{\ofb{0,1}}=D$ since $\bigcup_{j=0}^\infty A_j$ is dense in $D$ and $g_i$ is a surjection onto $\bigcup_{j=0}^iA_j$ by $(\rm{P-}\ref{l:factf:c:cont})_i$.
		Finally, observe that the sequence $\of{h\circ g_i}$ converges uniformly to $h\circ g$ since the
		sequence $\of{g_i}$ converges uniformly to $g$ and $h$ is continuous. But $h\circ g_i=f$ for all $i$. Consequently,
		$h\circ g=f$. This completes the proof of the lemma in the case when $D$ is not a tree.
		\vspace{0.25cm}
		
		\noindent
		\emph{Sketch of Proof in the case when $D$ is a tree.} In this case $D$ may be represented as a finite union $\bigcup_{i=0}^kA_i$, see \ref{p:AiTree}.
		Set $r_0$, $s_0$, $\tau_0$, $u_0$ and $v_0$ the same way as before and then construct $\of{r_i}_{i=1}^k$, $\of{s_i}_{i=1}^k$, $\of{\tau_i}_{i=1}^k$, $\of{u_i}_{i=1}^k$,  and $\of{v_i}_{i=1}^k$ satisfying conditions $(\ref{l:factf:cond1})_{i}$-$(\ref{l:factf:xinf(I)})_{i}$ for all  positive integers $i\le k$. Note that $(\ref{l:factf:diamtauli})_{i}$ and other estimates of distance by $2^{-i}$ are irrelevant in this finite case and may be omitted. After the $k$th step of the construction, define $h:D\to\ofb{0,1}$ by $h\of{x}=\tau_i{}^{-1}\of{x}$ for $x\in A_i$ for every nonnegative integer $i\le k$.  Then construct $g_0, g_1,\dots,g_k$ using the same recursive formula as above. Finally, set $g=g_k$ and observe that $h$ and $g$ defined this way satisfy the lemma.
	\end{proof}
	
	\section{A transitive map on $[0,1]$ with the small folds property}\label{s:transitive+sfp}
	W.R.R. Transue and the second author of the present paper constructed in \cite{minc-transue} a transitive map $f$ of $\ofb{0,1}$ onto itself such that $\varprojlim\of{\ofb{0,1},f}$ is homeomorphic to the pseudo-arc. It is possible, but not entirely clear that the map on $[0,1]$ constructed in \cite{minc-transue} has the small folds property. In this section we will tweak the original construction very slightly to be able to show that the small folds property is satisfied.
	
	\subsection{Summary of the original construction in \cite{minc-transue}}
	The two key elements of that construction are   \cite[Proposition 5, p. 1166]{minc-transue} and \cite[Lemma, p. 1167]{minc-transue}. The lemma is used repeatedly by the inductive construction in the proof of the main result in \cite{minc-transue} (Theorem on page 1169). In turn, the lemma uses Proposition 5 in each pass. We will summarize the proposition by briefly describing arguments passed to the routines and the output produced by them.
	
	\noindent $\bullet$ \cite[Proposition 5, p. 1166]{minc-transue}. \emph{Input}: positive numbers $\epsilon<1$ and $\gamma<\epsilon/4$. \emph{Output}: A piecewise linear continuous function $g$ mapping $\ofb{0,1}$ onto itself such that the distance between $g$ and the identity is estimated by $\epsilon$, and $g$ is $\gamma$-crooked between all $a,b\in\ofb{0,1}$ such that $\ofa{a-b}<\epsilon$. (See the original statement of the proposition in \cite{minc-transue} for more essential properties of $g$.)

	A continuous and piecewise linear function $f$ of $\ofb{0,1}$ onto itself is called admissible if $\ofa{f^{\prime}\of{t}}\ge 4$ for every t such that $f^{\prime}\of{t}$ exists and for every $0 \le a < b \le 1$ there is a positive integer $m$ such that $f^m\of{\ofb{a,b}}=\ofb{0,1}$. For example, the second iteration of the full tent map is admissible.
	
	\noindent $\bullet$ \cite[Lemma, p. 1167]{minc-transue}.  \emph{Input}: an admissible map $f$ and positive numbers $\eta$ and $\delta$.  \emph{Output}: A positive integer $n$ and an admissible map $F$ such that $f$ and $F$ are $\eta$ close, $F^n$ is $\delta$-crooked. Moreover, if $0\le a < b \le 1$ and $b - a\ge \eta$ , then
	$f\of{\ofb{a,b}}\subset F\of{\ofb{a,b}}$ and $F^n\of{\ofb{a,b}}=\ofb{0,1}$.
	
	In the proof of the lemma, properties of the input are used to select a positive number $\epsilon$, a positive integer $n$, and a positive number $\gamma$. (The order of this choice is important. The choice of $n$ depends on that of $\epsilon$. The choice of $\gamma$ depends on $\epsilon$ $n$.) Then \cite[Proposition 5, p. 1166]{minc-transue} is used to obtain $g$. The function $F=f\circ g$ satisfies the lemma.
	
	In the following observation we use the same notation as in the first three lines of page 1168 in \cite{minc-transue}.
	\begin{observation}\label{o:eps&gaminlemma}
		In the proof of \cite[Lemma, p. 1167]{minc-transue}, $\epsilon$ is selected to be exactly $\eta/s$ where $s$ is such that $\ofa{f^{\prime}\of{t}}<s$ whenever $f^{\prime}\of{t}$  exists. In fact, we may set $\epsilon$ to be any positive number $\le\eta/s$ and apply the same proof as it is written in \cite{minc-transue} without any need for an additional change. Another degree of freedom in the proof of the lemma is the choice of $\gamma$. After $\epsilon$ and $n$ are selected, $\gamma$ may be chosen to be any positive number less than $\min\of{\alpha,s^{-n},\epsilon/4,\delta s^{-n}/5}$. This will allow us to strengthen the lemma by imposing an additional condition on $\gamma$.
	\end{observation}
	
	In the proof of the main result in \cite{minc-transue} (Theorem on page 1169) a sequence of admissible functions $f_1,f_2,\dots$ and a sequence of positive integers $n\of{1},n\of{2},\dots$ are constructed by induction to satisfy the following three conditions
	\begin{list}{(\roman{lcount})}{\usecounter{lcount}}
		\item $\ofa{f_{i+1}\of{t}-f_i\of{t}}<2^{-i}$ for each $t\in\ofb{0,1}$,
		\item $f_i^{n\of{k}}$ is $\of{2^{-k}-2^{-k-i}}$-crooked for each positive integer $k\le i$, and
		\item if $0\le a<b\le1$ and $b-a\ge 2^{-k}$, then $f_i^{n\of{k}}\of{\ofb{a,b}}=\ofb{0,1}$ for each positive integer $k\le i$.
	\end{list}
	
	For each integer $i\ge 2$, \cite[Lemma, p. 1167]{minc-transue} is used with $f=f_{i-1}$ and with a certain choice of $\eta$ and $\delta$. Then $n\of{i}$ and $f_i$ are defined by  setting $n\of{i}=n$ and $f_i=F$ where $n$ and $F$ are output by the lemma.
	
	The first condition in the construction guarantees that the sequence $\of{f_i}$ converges uniformly. The second condition implies that the inverse limit of copies of $\ofb{0,1}$ with $\lim_{i\to\infty}f_i$ as the bonding map is the pseudo-arc. Finally, $\lim_{i\to\infty}f_i$ is transitive by (iii) and Theorem 6 of \cite{barge-martin}.
	
	\subsection{Adjustments to the construction}
	We will use Observation \ref{o:eps&gaminlemma} to obtain the following lemma.
	\begin{lemma}[{Replacement for \cite[Lemma, p. 1167]{minc-transue}}]\label{l:repl}
		Let $f:\ofb{0,1}\to\ofb{0,1}$ be an admissible map. Let $\eta$, $\delta$ and $\lambda$ be three positive numbers. Then there is an integer $n$ and there are continuous maps $g$ and $F$ of $\ofb{0,1}$ onto itself satisfying the following conditions:
		\begin{enumerate}
			\item $F=f\circ g$,
			\item \label{l:repl:F-f g-id} $\ofa{F\of{t}-f\of{t}}<\eta$ and $\ofa{g(t)-t}<\eta$ for each $t\in\ofb{0,1}$,
			\item $F^n$ is $\delta$-crooked,
			\item \label{l:repl:fjinFj} if $0\le a<b\le 1$ and $b-a\ge\eta$, then $f^j\of{\ofb{a,b}}\subset F^j\of{\ofb{a,b}}$ for each positive integer $j$,
			\item if $0\le a<b\le 1$ and $b-a\ge\eta$, then $F^n\of{\ofb{a,b}}=\ofb{0,1}$,
			\item $F$ is admissible, and
			\item \label{l:repl:lambda}
			there exist positive numbers $\beta<\lambda$ and $\xi<\beta/4$ satisfying the following condition:
			for every $a$ and $b$ such that $\ofa{a - b} < \beta$, $F$ is $\xi$-crooked between $a$ and $b$.
		\end{enumerate}
	\end{lemma}
	\begin{proof}
		Let $\alpha$ and $s$ be defined in the same way as in the proof of \cite[Lemma, p. 1167]{minc-transue}. Let $\epsilon$ be a positive number less than $\min\of{\eta/s,\alpha}$.
		It follows from (1) on page 1167 in  \cite[Lemma, p. 1167]{minc-transue} that
		\begin{observationInTh}\label{l:repl:obs}
			Suppose that $a,b,a^{\prime},b^{\prime}\in\ofb{0,1}$ are such that $\ofa{a-b}<2\epsilon$ and $\ofb{a^{\prime},b^{\prime}}\subset f{}^{-1}\of{\ofb{a,b}}$. Then $\ofa{a^{\prime}-b^{\prime}}<\epsilon$.
		\end{observationInTh}
		Let $n$ be defined in the same way as in the proof of Lemma in \cite{minc-transue}, that is if $0\le a<b\le 1$ and $b-a>\epsilon/4$, then $f^n\of{\ofb{a,b}}=\ofb{0,1}$.
		Let $\beta$ be a positive number less than $\min\of{2\epsilon,\lambda}$, let Let $\xi$ be a positive number less than $\beta/4$, and let $\gamma$ be a positive number less than $\min\of{\alpha,s^{-n},\epsilon/4,\delta s^{-n}/5,\xi/s}$. As it was done in the original proof we now use \cite[Proposition 5]{minc-transue} to get the map $g$ and define $F=f\circ g$. (Notice that $\ofa{g(t)-t}$ could be estimated in (\ref{l:repl:F-f g-id}) by $\eta/s$ instead just by $\eta$ as it is stated in that condition.) The proof of all conditions except for (\ref{l:repl:fjinFj}) and (\ref{l:repl:lambda}) was given in \cite{minc-transue} and will be omitted here. We will only prove (\ref{l:repl:fjinFj}) and (\ref{l:repl:lambda}).
		
		\begin{proof}[Proof of (\ref{l:repl:fjinFj})]\renewcommand{\qedsymbol}{}
			Recall that the number $s$ was defined in \cite{minc-transue} such that $\ofa{f^{\prime}\of{t}}<s$ for all $t$ such that $f^{\prime}\of{t}$ is defined. Observe that $s>4$ because $f$ is admissible. It was observed in \cite{minc-transue} that $\operatorname{diam}\of{f\of{C}}\le s\operatorname{diam}\of{C}$ for every $C\subset\ofb{0,1}$; see (2) on page 1167 in \cite{minc-transue}.
			
			Let $a$ and $b$ be such that $0\le a<b\le 1$ and $b-a\ge\eta$. Since $\epsilon<\eta/s<\eta/4$, $b-a\ge4\epsilon>\epsilon/4$. It follows from the choice of $n$ that $f^n\of{\ofb{a,b}}=\ofb{0,1}$. Observe that
			\begin{equation*}\label{l:repl:eq1}
				\operatorname{diam}\of{f^j\of{\ofb{a,b}}}\ge \gamma \quad \text{for each nonnegative integer $j$.} \tag{$*$}
			\end{equation*}
			Otherwise, $\operatorname{diam}\of{f^{j+n}\of{\ofb{a,b}}}\le s^n\operatorname{diam}\of{f^{j}\of{\ofb{a,b}}}<s^n\gamma$ which is a contradiction because $f^{j+n}\of{\ofb{a,b}}=\ofb{0,1}$ and $s^n\gamma<s^ns^{-n}=1$.
			
			\cite[Proposition 5 (v)]{minc-transue} states that $A\subset g\of{A}$ for each interval $A\subset\ofb{0,1}$ such that $\operatorname{diam}\of{A}\ge\gamma$. Applying $f$ to both sides of the inclusion $A\subset g\of{A}$ we get  $f\of{A}\subset f\circ g\of{A}=F\of{A}$. Hence,
			\begin{equation*}\label{l:repl:eq2}
				f\of{A}\subset F\of{A} \quad \text{for each interval $A\subset\ofb{0,1}$ such that $\operatorname{diam}\of{A}\ge\gamma$.}\tag{$**$}
			\end{equation*}
			We will prove the inclusion
			\begin{equation*}
				f^{j}\of{\ofb{a,b}}\subset F^{j}\of{\ofb{a,b}} \tag{I$_j$}
			\end{equation*}
			by induction with respect to $j$.  It follows from \eqref{l:repl:eq1} for $j=0$ that $b-a\ge\gamma$. So we  may use \eqref{l:repl:eq2} with $A=\ofb{a,b}$ to get
			\thetag{I$_1$}. Now, suppose $j\ge2$ and \thetag{I$_{j-1}$} is true. We need to show \thetag{I$_{j}$}.
			Applying $f$ to both sides of the inequality \thetag{I$_{j-1}$} we infer that $f^{j}\of{\ofb{a,b}}\subset f\of{F^{j-1}\of{\ofb{a,b}}}$.
			Since it follows from \eqref{l:repl:eq1} for $j-1$ and \thetag{I$_{j-1}$} that $\operatorname{diam}\of{F^{j-1}\of{\ofb{a,b}}}\ge\gamma$, we  may use \eqref{l:repl:eq2} with $A=F^{j-1}\of{\ofb{a,b}}$ to get $f\of{F^{j-1}\of{\ofb{a,b}}}\subset F\of{F^{j-1}\of{\ofb{a,b}}}=F^{j}\of{\ofb{a,b}}$. Hence,
			$$f^{j}\of{\ofb{a,b}}\subset f\of{F^{j-1}\of{\ofb{a,b}}}\subset F^{j}\of{\ofb{a,b}}.$$
			So, \thetag{I$_{j}$} is true and the proof of
			(\ref{l:repl:fjinFj}) is complete.
		\end{proof}
		
		\begin{proof}[Proof of (\ref{l:repl:lambda})]\renewcommand{\qedsymbol}{}
			Take any $a$ and $b$ such that $\ofa{a - b} < \beta$. We need to show that $F=f\circ g$ is $\xi$-crooked between $a$ and $b$.
			Take $c,d\in\ofb{0,1}$ such that $f\circ g\of{c}=a$ and $f\circ g\of{d}=b$. Let $c_0$ be the last point in $\ofb{c,d}$ such that $f\circ g\of{c_0}=a$. Clearly, $c_0\in[c,d)$.
			Let $d_0$ be the first point in $\ofb{c_0,d}$ such that $f\circ g\of{d_0}=b$. Clearly, $f\circ g\of{\ofb{c_0,d_0}}=\ofb{a,b}$ and $f\circ g\of{\of{c_0,d_0}}=\of{a,b}$.
			Consequently, $g\of{\ofb{c_0,d_0}}=\ofb{g\of{c_0},g\of{d_0}}$ and $g\of{\of{c_0,d_0}}=\of{g\of{c_0},g\of{d_0}}$.
			Since $\ofa{a - b} < \beta<2\epsilon$, it follows from \ref{l:repl:obs} that $\ofa{g\of{c_0}-g\of{d_0}}<\epsilon$. By \cite[Proposition 5 (ii)]{minc-transue}, $g$ is $\gamma$-crooked between $g\of{c_0}$ and $g\of{d_0}$. So, there exists $c^{\prime}$ between $c_0$ and $d_0$, and there exists $d^{\prime}$ between $c^{\prime}$ and $d_0$ such that $\ofa{g\of{d_0}-g\of{c^{\prime}}}<\gamma$ and $\ofa{g\of{c_0}-g\of{d^{\prime}}}<\gamma$.
			It follows from the choice of $c_0$ and $d_0$ that $c^{\prime}$ is between $c$ and $d$, and $d^{\prime}$ is between $c^{\prime}$ and $d$.
			Since $\operatorname{diam}\of{f\of{C}}<s\operatorname{diam}\of{C}$ for every nonempty set $C\subset\ofb{0,1}$ by \cite[(2), p. 1167]{minc-transue}, $\gamma<\xi/s$, $f\circ g\of{d_0}=b$ and $f\circ g\of{c_0}=a$ we infer that
			$\ofa{b-f\circ g\of{c^{\prime}}}<\xi$ and
			$\ofa{a-f\circ g\of{d^{\prime}}}<\xi$. Thus, $F=f\circ g$ is $\xi$-crooked between $a$ and $b$.
		\end{proof}
	\end{proof}
	
	\begin{proposition}\label{p:crookcomp}
		Let $f$ and $g$ be continuous functions of $\ofb{0,1}$ into $\ofb{0,1}$. Suppose that $f$ is $\xi$-crooked between $a$ and $b$ for some $a,b\in\ofb{0,1}$ and a positive number $\xi$. Then  $f\circ g$ is also $\xi$-crooked between $a$ and $b$.
	\end{proposition}
	\begin{proof}
		Suppose there are $c,d\in\ofb{0,1}$ such that $f\circ g\of{c}=a$ and $f\circ g\of{d}=b$.
		Since $f$ is $\xi$-crooked between $a$ and $b$ for some $a,b\in\ofb{0,1}$, there is a point $c_1$ between $g\of{c}$ and $g\of{d}$, and there is a point $d_1$ between $c_1$ and $g\of{d}$ such that $\ofa{b-f\of{c_1}}\le\xi$ and $\ofa{a-f\of{d_1}}\le\xi$. Since $g$ is continuous, there is a point $c^{\prime}$ between $c$ and $d$, and there is a point $d^{\prime}$ between  $c^{\prime}$ and $d$ such that $g\of{c^{\prime}}=c_1$ and $g\of{d^{\prime}}=d_1$. Observe that $\ofa{b-f\circ g\of{c^{\prime}}}=\ofa{b-f\of{c_1}}\le\xi$ and $\ofa{a-f\circ g\of{d^{\prime}}}=\ofa{a-f\of{d_1}}\le\xi$.
	\end{proof}	
	
	\begin{proposition}\label{p:compg}
		Let $\of{g_i}_{j=1}^\infty$ be a sequence of continuous functions of $\ofb{0,1}$ into $\ofb{0,1}$. For all integers $i$ and $j$ such that $1\le i< j$, let $g_{i,j}$ denote the composition $g_i\circ g_{i+1}\circ\dots g_j$. Additionally, set $g_{i,i}=g_i$. Suppose that, for each nonnegative integer $i$, the sequence $\of{g_{i,j}}_{j=i}^\infty$ uniformly converges. Let $g_{i,\infty}=\lim_{j\to\infty}g_{i,j}$. Then $g_{i,j}\circ g_{j+1,\infty}=g_{i,\infty}$ for all positive integers $i$ and $j$ such that $i\le j$.
	\end{proposition}
	
	In the next proposition we will use the same notation as in the previous one.
	\begin{proposition}\label{p:smfolds} Let $\of{g_i}_{j=1}^\infty$ be a sequence with the same properties as in Proposition \ref{p:compg}.
		Suppose also that for each $\lambda>0$ there is a positive integer $j$, and there exist positive numbers $\beta<\lambda$ and $\xi<\beta/4$ satisfying the following condition
		\begin{equation*}\label{p:smfolds:*}
			\text{for every $a$ and $b$ such that\ } \ofa{a - b} < \beta \text{,\ $g_{1,j}$ is $\xi$-crooked between $a$ and $b$.} \tag{$*_{1,j}$}
		\end{equation*}
		Then $g_{1,\infty}$ has the small folds property.
	\end{proposition}
	\begin{proof}
		In order to prove the proposition it is enough to show \thetag{$*_{1,\infty}$} that is \eqref{p:smfolds:*} with $g_{1,j}$ replaced by $g_{1,\infty}$.
		For that purpose observe that $g_{1,\infty}=g_{1,j}\circ g_{j+1,\infty}$ by \ref{p:compg}.
		Now, use Proposition \ref{p:crookcomp} with $f=g_{1,j}$ and $g=g_{1,\infty}$.
	\end{proof}
	
	The following proposition is well-known. We state it here for convenience. Note that $F$ in this proposition does not have to be continuous. Also, a similar proposition with $\ofb{0,1}$ replaced by an arbitrary compact metric space is true.
	
	\begin{proposition}\label{p:f^n-F^n}
		Suppose $n$ is a positive integer and $f:\ofb{0,1}\to\ofb{0,1}$ be a continuous function. Then, for each $\epsilon>0$ there exists $\eta>0$ with the property $\ofa{f^n\of{t}-F^n\of{t}}<\epsilon$ for all $t\in\ofb{0,1}$ and each function $F:\ofb{0,1}\to\ofb{0,1}$ such that $\ofa{f\of{t}-F\of{t}}<\eta$ for all $t\in\ofb{0,1}$.
	\end{proposition}
	\begin{proof}
		The proposition is trivial if $n=1$. Suppose $n>1$ and the proposition is true for $n-1$. We will prove that it is also true for $n$.
		
		Take an arbitrary $\epsilon>0$. Since $f$ is continuous, there is $\delta>0$ such that $\ofa{f\of{a}-f\of{b}}<\epsilon/2$ for all $a,b\in\ofb{0,1}$ such that $\ofa{a-b}<\delta$.
		Using the proposition with $n-1$ and $\epsilon$ replaced by $\delta$, we infer that there is a positive number $\eta\le\epsilon/2$ with the property
		$\ofa{f^{n-1}\of{t}-F^{n-1}\of{t}}<\delta$ for all $t\in\ofb{0,1}$ and each function $F:\ofb{0,1}\to\ofb{0,1}$ such that $\ofa{f\of{t}-F\of{t}}<\eta$ for all $t\in\ofb{0,1}$. Suppose $F$ is a specific function such that $\ofa{f\of{t}-F\of{t}}<\eta$ for all $t\in\ofb{0,1}$. In particular, $\ofa{f\of{F^{n-1}\of{t}}-F\of{F^{n-1}\of{t}}}<\eta\le\epsilon/2$ for all $t\in\ofb{0,1}$. It follows from the choices of $\eta$ and $\delta$ that
		$\ofa{f\of{f^{n-1}\of{t}}-f\of{F^{n-1}\of{t}}}<\epsilon/2$ for all $t\in\ofb{0,1}$. Consequently,
		$\ofa{f^{n}\of{t}-F^{n}\of{t}}=\ofa{f^{n}\of{t}-f\of{F^{n-1}\of{t}}+f\of{F^{n-1}\of{t}}-F^{n}\of{t}}\le
		\ofa{f\of{f^{n-1}\of{t}}-f\of{F^{n-1}\of{t}}}+\ofa{f\of{F^{n-1}\of{t}}-F\of{F^{n-1}\of{t}}}<\epsilon/2+\epsilon/2=\epsilon$ for all $t\in\ofb{0,1}$.
	\end{proof}
	
	\begin{theorem}\label{t:smfolds}
		There is a map $f:\ofb{0,1}\to\ofb{0,1}$ such that
		\begin{enumerate}
			
			\item\label{t:smfolds:ps} the inverse limit of copies of $\ofb{0,1}$ with $f$ as the bonding map is a pseudo-arc,
			\item\label{t:smfolds:tr} $f$ is topologically exact, and
			\item\label{t:smfolds:sfp} $f$ has the small folds property.
		\end{enumerate}
	\end{theorem}
	\begin{proof}
		The proof of this theorem is very similar to that of \cite[Theorem, p 1169]{minc-transue}. As it was done in \cite{minc-transue}, we construct a sequence of positive integers $n\of{1}, n\of{2}, \dots$ and a sequence of admissible functions $f_1, f_2, \dots$ of $\ofb{0,1}$ onto itself. In \cite{minc-transue}, the lemma was used with $f=f_{i-1}$ to define $f_i$ as $F=f\circ g$ for $i\ge 2$. We will use here Lemma \ref{l:repl} instead and remember $g$ as $g_i$ for future use.  So, we will also construct another sequence of continuous functions $g_2,g_3,\dots$ of $\ofb{0,1}$ onto itself. Additionally, we set $g_1=f_1$. This allows us to use the notation from Proposition  \ref{p:compg}. In particular, $f_i=g_{1,i}$ for each integer $i$.
		
		Our construction will have the following properties for each positive integer $i$:
		\begin{list}{(\roman{lcount})}{\usecounter{lcount}}
			\item if $i>1$, then $\ofa{g_{k,i-1}\of{t}-g_{k,i}\of{t}}<2^{-i}$ for each $t\in\ofb{0,1}$ and each positive integer $k\le i-1$,
			\item $f_i^{n\of{k}}$ is $\of{2^{-k}-2^{-k-i}}$-crooked for each positive integer $k\le i$,
			\item if $0\le a<b\le1$ and $b-a\ge 2^{-k}$, then $f_i^{n\of{k}}\of{\ofb{a,b}}=\ofb{0,1}$ for each positive integer $k\le i$, and
			\item there are positive numbers $\beta<2^{-i}$ and $\xi<\beta/4$ satisfying the condition:
			for every $a$ and $b$ such that $\ofa{a - b} < \beta$, $g_1^i=f_i$ is $\xi$-crooked between $a$ and $b$.
		\end{list}
		
		To construct $n\of{1}$ and $f_1$, we use Lemma \ref{l:repl} with any admissible map $f$, $\eta=1/2$, $\delta=1/4$ and $\lambda=1/2$. Then we set $n\of{1}=n$, $f_1=F$ and $g_1=F$ where $n$ and $F$ are from the lemma. We assume that $n\of{1},\dots,n\of{i-1}$, $f_1,\dots,f_{i-1}$ and $g_1,\dots,g_{i-1}$ have already been constructed for some integer $i\ge2$. We will construct $n\of{i}$, $f_i$ and $g_i$.
		
		Since each of the functions  $g_1,\dots,g_{i-1}$ is continuous, there is a positive number $\eta^{\prime}$ with the property that if $g:\ofb{0,1}\to\ofb{0,1}$ is a function such that $\ofa{g\of{t}-t}<\eta^{\prime}$ for all $t\in\ofb{0,1}$, then $\ofa{g_{k,i-1}\of{t}-g_{k,i-1}\circ g\of{t}}<2^{-i}$ for each positive integer $k\le i-1$ and all $t\in\ofb{0,1}$.
		
		For each positive integer $k\le i-1$, use Proposition \ref{p:f^n-F^n} with $n=n\of{k}$, $f=F_{i-1}$ and $\epsilon=2^{-k-i-1}$ to get a positive number $\eta_k$ with the property
		\begin{equation*}\label{t:smfolds:eq}
			\ofa{f_{i-1}^{n\of{k}}\of{t}-F^{n\of{k}}\of{t}}<2^{-k-i-1} \quad \text{for all $t\in\ofb{0,1}$} \tag{$*$}
		\end{equation*}
		and each function $F:\ofb{0,1}\to\ofb{0,1}$ such that $\ofa{f_{i-1}\of{t}-F\of{t}}<\eta_k$ for all $t\in\ofb{0,1}$.
		Observe that it follows from condition (ii) for $i-1$, \eqref{t:smfolds:eq} and \cite[Proposition 2]{minc-transue} that
		\begin{equation*}\label{t:smfolds:eq2}
			F^{n\of{k}} \text{ is $(2^{-k}-2^{-k-i})$-crooked}. \tag{$**$}
		\end{equation*}
		
		Let $\eta$ be a positive number less than $\min\of{2^{-i},\eta^{\prime},\eta_1,\eta_2,\dots,\eta_{i-1}}$. Now we use Lemma \ref{l:repl} with $\eta$ we defined, $f=f_{i-1}$, $\delta=2^{-i}-2^{-i-i}$ and $\lambda=2^{-i}$. Then we set $n\of{i}=n$, $f_i=F$ and $g_i=g$ where $n$, $F$ and $g$ are obtained from the lemma. Clearly, $f_i=f_{i-1}\circ g_i$ and $f_i=g_{1,i}$. Observe that (i) is satisfied since $\eta<\eta^{\prime}$. Condition (ii) follows from \eqref{t:smfolds:eq2} since $\eta<\eta_k$ for each positive integer $k\le i-1$. To prove (iii), it is enough to observe that if $b-a\ge 2^{-i}>\eta$, then $f_{i-1}^j\of{\ofb{a,b}}\subset f_i^j\of{\ofb{a,b}}$ for each positive integer $j$, see \ref{l:repl} (\ref{l:repl:fjinFj}). Finally, (iv) follows from \ref{l:repl} (\ref{l:repl:lambda}) since $\lambda=2^{-i}$.
		
		By (i), the sequence $\of{g_{i,j}}_{j=i}^\infty$ converges uniformly for each positive integer $i$. In particular, $\of{g_{1,j}}_{j=1}^\infty=\of{f_j}_{j=1}^\infty$ converges uniformly. Denote its limit by $f$.  Our proof of \ref{t:smfolds} (\ref{t:smfolds:ps}) and \ref{t:smfolds} (\ref{t:smfolds:tr}) exactly follows \cite{minc-transue}.
		Applying Propositions 1 and 3 in \cite{minc-transue}, we infer that $f^{n\of{k}}$ is $\of{2^{-k}}$-crooked for each positive integer $k$. Applying Propositions 1 and 4 in \cite{minc-transue}, we get the result that the inverse limit of copies of $\ofb{0,1}$ with $f$ as the bonding map is a pseudo-arc.
		Condition (iii) of the construction implies that if $0\le a<b\le1$ and $b-a\ge2^{-k}$, then $f^{n\of{k}}\of{\ofb{a,b}}=\ofb{0,1}$. It follows that $f$ is topologically exact. Since the sequence $\of{g_{i,j}}_{j=i}^\infty$ converges uniformly for each positive integer $i$, condition (iv) of the construction allows us to use Proposition  \ref{p:smfolds} and get the result that $f$ has the small folds property.
	\end{proof}
	
\begin{theorem}\label{t:mainfactorization}
	There exists a topologically mixing map $f$ of $\ofb{0,1}$ onto itself such that the inverse limit space $\varprojlim\of{\ofb{0,1},f}$ is the pseudo-arc, and for any nondegenerate dendrite $D$, there exist onto maps $g:\ofb{0,1}\to D$ and $h:D\to\ofb{0,1}$ such that $h\circ g=f$. Moreover, the map $F=g\circ h$ of $D$ onto itself is topologically mixing, the natural extensions of $f$ and $F$ are conjugate, and the inverse limit space $\varprojlim\of{D,F}$ is the pseudo-arc.
	\end{theorem}
\begin{proof}
  The theorem follows easily from Lemma \ref{l:factf}, Theorem \ref{t:smfolds} and Proposition \ref{p:hg-gh}.
\end{proof}
Our construction gives in fact the following stronger result.
\begin{theorem}\label{t:mfk}
	There exists a topologically mixing map $f$ of $\ofb{0,1}$ onto itself such that the inverse limit space $\varprojlim\of{\ofb{0,1},f}$ is the pseudo-arc, and for any $k\in\mathbb{N}$ and any nondegenerate dendrites $D_1,\ldots,D_k$, there exist onto maps $g_i:\ofb{0,1}\to D_i$ and $h_i:D_i\to\ofb{0,1}$, for $i=1,\ldots,k$, such that $h_i\circ g_i=f$. Moreover, the map $F_i = g_i \circ \cdots \circ h_i$ of $D_i$ onto itself is topologically mixing, the natural extensions of $f$ and $F_i$ are conjugate, and the inverse limit space $\varprojlim\of{D_i,F_i}$ is the pseudo-arc, for $i=1,\ldots,k$.
\end{theorem}

	\begin{equation*}
	\begin{tikzcd}[baseline=\the\dimexpr\fontdimen22\textfont2\relax]
		&  & I \ar[ld,"g_1"] & I \ar[l,dashed,"f = h_2 \circ g_2" above] \ar[ld,"g_2"] & I \ar[l,dashed,"f = h_1 \circ g_1" above] \ar[ld,"g_1"] & 
		I \ar[l,dashed,"f = h_2 \circ g_2" above] \ar[ld,"g_2"] & I \ar[l,dashed,"f = h_1 \circ g_1" above] \ar[ld,"g_1"] & \ar[l, dotted]\\
		& D_1 & D_2 \ar[u,"h_2" right] \ar[l,dashed,"g_1\circ h_2"] & D_1 \ar[u,"h_1" right] \ar[l,dashed,"g_2\circ h_1"] \ar[ll, bend left=40,"F_1 = g_1 \circ h_2 \circ g_2 \circ h_1"] 
		& D_2 \ar[u,"h_2" right]\ar[l,dashed,"g_1\circ h_2"]& D_1 \ar[u,"h_1" right]\ar[l,dashed,"g_2\circ h_1"] \ar[l, dotted] \ar[ll, bend left=40,"F_1 = g_1 \circ h_2 \circ g_2 \circ h_1"] & \ar[l, dotted]\\
	\end{tikzcd}
\end{equation*}

\vspace{0.5cm}
	
	\bibliographystyle{amsplain}

\begin{thebibliography}{99}
		\bibitem{OprochaD}
		 {\sc G. Acosta, R. Hernández-Gutiérrez, I. Naghmouchi, P. Oprocha} {\em Periodic points and transitivity on dendrites.} {\bf Ergodic Theory Dynam. Systems 37} (2017), 2017–2033.
		
		 \bibitem{Stu} {\sc S. Baldwin,} {\em Entropy estimates for transitive maps on trees.} {\bf Topology 40} (2001),  551–569.
		
		\bibitem{barge-martin} {\sc M. Barge and J. Martin}, \textit{Dense orbits on the interval}, {\bf Michigan Math. J.} 34(1987), 3--11.
		
		\bibitem{Barge} {\sc M. Barge}, {\it Homoclinic intersections and indecomposability.}
		{\bf Proc. Amer. Math. Soc. 101} (1987), no. 3, 541–544.
		
		\bibitem{Barge1} {\sc M.\ Barge, S. Holte} \emph{Nearly one-dimensional H\'enon attractors and inverse limits},
		 {\bf Nonlinearity 8} (1995), 29-42.
		
		\bibitem{Barge2} {\sc M. Barge, B. Diamond,} {\it Stable and unstable manifold structures in the H\'enon family.} {\bf Ergodic Theory Dynam. Systems 19 }(1999),  309--338.
		
		\bibitem{Bing} {\sc R. H. Bing}, {\it A homogeneous indecomposable plane continuum.} {\bf Duke Math. J. 15} (1948), 729--742.
		
		 \bibitem{Blokh} {\sc A. M. Blokh}, {\it The “spectral” decomposition for one-dimensional maps.} In Dynamics Reported, Dynam. Report. Expositions Dynam. Systems (N. S.), no. 4, pages 1–59. Springer, Berlin, 1995.
		 
		\bibitem{BOBirk} {\sc J. Boro\'nski, P. Oprocha} \textit{Rotational chaos and strange attractors on the 2-torus}, \textbf{Math. Z. 279} (2015), 689--702.
		
		\bibitem{BO} {\sc J. P. Boro\'nski, P. Oprocha } \textit{On Entropy of Graph Maps That Give Hereditarily Indecomposable Inverse Limits}, \textbf{Journal of Dynamics and Differential Equations 29} (2017), 685--699.
		
		\bibitem{boronski-stimac} {\sc J. Boro\'nski and S. \v Stimac}, \textit{Densely branching trees as models for Henon-like and Lozi-like attractors},  	arXiv:2104.14780
		
		\bibitem{BoylandBLMS} {\sc P. Boyland, A. de Carvalho, T. Hall,} \textit{Inverse limits as attractors in parametrized families}, {\bf Bull.\ Lond.\ Math.\ Soc.\ 45(5)} (2013), 1075--1085.	
		
		\bibitem{BoylandInventiones} {\sc P. Boyland, A. de Carvalho, T. Hall,}  {\it New rotation sets in a family of torus homeomorphisms.} {\bf Invent. Math.} 204 (2016), no. 3, 895--937.
		
		\bibitem{BoylandGT} {\sc P. Boyland, A. de Carvalho, T. Hall,} {\em Natural extensions of unimodal maps: prime ends of planar embeddings and semi-conjugacy to sphere homeomorphisms},  {\bf Geom. Topol. 25} (2021), no. 1, 111--228.
		
        \bibitem{Bruin}  {\sc H. Bruin,} {\it Planar embeddings of inverse limit spaces of unimodal maps.} {\bf Topology Appl. 96} (1999), 191--208.
		
		\bibitem{KwietniakD} {\sc J. Byszewski, F. Falniowski, D. Kwietniak, }{\it Transitive dendrite map with zero entropy.} 
		{\bf Ergodic Theory Dynam. Systems 37} (2017), 2077--2083.
		
		\bibitem{Cheritat}{\sc Ch\'eritat, A.} \textit{Relatively compact Siegel disks with non-locally connected boundaries.} \textbf{Math. Ann. 349} (2011), 529--542.
		
		\bibitem{CO} {\sc J. \v Cin\v c, P. Oprocha} {\it Pseudo-arc as the attractor in the disc -- topological and measure-theoretic aspects}, arXiv:2107.10347
		
		\bibitem{Drwiega}  {\sc T. Drwiega, P. Oprocha,} {\it Topologically mixing maps and the pseudoarc.}  {\bf Ukrainian Math. J. 66} (2014), 197--208.
		
		\bibitem{engelking} {\sc R. Engelking}, \textit{General topology,} Sigma Series in Pure Mathematics, 6. Heldermann Verlag, Berlin, 1989.
		
		\bibitem{Handel} {\sc M. Handel} \textit{A pathological area preserving $C^\infty$ diffeomorphism of the plane.} \textbf{Proc. Amer. Math. Soc. 6}  (1982), 163--168.
		
		\bibitem{Herman} {\sc M.-R. Herman} {\em Construction of some curious diffeomorphisms of the Riemann sphere.} {\bfseries J. London Math. Soc. 34} (1986), 375--384.
		
    	\bibitem{HoehnD} {\sc L.C. Hoehn, C. Mouron}, {\it Hierarchies of chaotic maps on continua.} {\bf Ergodic Theory Dynam. Systems 34} (2014), 1897--1913.
		
		\bibitem{HO} {\sc L.C. Hoehn, L.G. Oversteegen}, {\it A complete classification of homogeneous plane continua.} {\bf Acta Math. 216} (2016), 177--216.
		
		\bibitem{HO2} {\sc Hoehn, L. C., Oversteegen, L. G.} \textit{A complete classification of hereditarily equivalent plane continua.}  {\bf Adv. Math. 368} (2020), 107131, 8 pp.
		
		\bibitem{Kawamura} {\sc K. Kawamura, H. M. Tuncali, and E. D. Tymchatyn,} {\it Hereditarily indecomposable inverse limits of graphs,} {\bf Fund. Math. 185} (2005), 195--210.
		
		\bibitem{KoscielniakOprochaTunchali}  {\sc P. Ko\'scielniak, P. Oprocha, M. Tuncali,} {\it Hereditarily indecomposable inverse limits of graphs: shadowing, mixing and exactness.} {\bf Proc. Amer. Math. Soc. 142} (2014), 681--694.
		
		\bibitem{Kwapisz} {\sc J. Kwapisz,} {\it A toral diffeomorphism with a nonpolygonal rotation set.} {\bf Nonlinearity 8} (1995), 461--476.	

		\bibitem{Li} {\sc S. Li}, {\it Dynamical properties of the shift maps on the inverse limit spaces}, {\bf Ergodic Theory Dynam. Systems 12} (1992) 95--108.
		
		\bibitem{MouronMix} {\sc V. Martínez-de-la-Vega, J. M. Martínez-Montejano, C. Mouron} \textit{Mixing homeomorphisms and indecomposability}, {\bf Topology Appl. 254} (2019), 50-58.
		
		\bibitem{minc-transue} {\sc P. Minc and W.R.R. Transue}, \textit{ A transitive map on $[0,1]$ whose inverse limit is the pseudoarc}, {\bf Proc. Amer. Math. Soc. 111} (1991), 1165-–1170.
		
		\bibitem{Moise} {\sc E. E. Moise}, {\it An indecomposable plane continuum which is homeomorphic to each of its nondegenerate subcontinua.}
		{\bf Trans. Amer. Math. Soc. 63} (1948), 581–594.
		
		\bibitem{Mouron} {\sc C. Mouron} \textit{Entropy of shift maps of the pseudo-arc}, {\bf Topology Appl. 159} (2012), 34--39.	
		
		\bibitem{Nadler} {\sc Sam B. Nadler, Jr.}, \textit{Continuum Theory: An Introduction,} Monographs and Textbooks in Pure and Applied Mathematics, 158. Marcel Dekker, Inc., New York, 1992.
		
		\bibitem{Rempe-Gillen} {\sc L. Rempe-Gillen} \textit{Arc-like continua, Julia sets of entire functions, and Eremenko's Conjecture}, arXiv:1610.06278v4.
		
		\bibitem{Rohlin} {\sc V.  A.  Rohlin},  {\it Lectures  on  the  entropy  theory  of  transformations  with  invariant measure.} {\bf Uspehi Mat. Nauk 22} (1967), 3--56.
		
		\bibitem{Paco} {\sc F. R. Ruiz del Portal,} {\it Stable sets of planar homeomorphisms with translation pseudo-arcs.} {\bf Discrete Contin. Dyn. Syst. Ser. S 12} (2019), no. 8, 2379–2390.
		
		\bibitem{Wazewski} {\sc T. Wazewski,} {\em Sur les courbes de Jordan ne renfermant aucune courbe simple ferme\'e de Jordan.} {\bf Ann. Soc. Polonaise Math. 2} (1923), 49--170.
		
		\bibitem{WilliamsUn}
		{\sc R.F.\ Williams}, \emph{ A note on unstable homeomorphisms}.
		{\bf Proc. Amer. Math. Soc. 6} (1955), 308--309.
		\bibitem{Williams}
		{\sc R.F.\ Williams}, \emph{One-dimensional non-wandering sets}, {\bf Topology 6} (1967), 473-487.
		
		\bibitem{Whyburn} {\sc G. T. Whyburn}, \textit{Analytic Topology,} Amer. Math. Soc. Colloq. Publ., vol. 28, Amer. Math. Soc., Providence, R.I., 1942.
	\end{thebibliography}

\end{document}